\documentclass[11pt,letterpaper,reqno]{amsart}

\usepackage{tikz}
\usetikzlibrary{positioning, shapes.geometric, arrows.meta}
\usepackage{amssymb}
\usepackage{amsmath}
\usepackage{amsthm}
\usepackage{amsfonts}
\usepackage{bbm}
\usepackage{enumitem}
\usepackage{pgfplots}
\pgfplotsset{compat=1.18}
\usepackage{booktabs}

\usepackage{graphicx}
\usepackage[T1]{fontenc}
\usepackage{doi}
\usepackage{float}
\usepackage{bookmark}
\usepackage{hyperref}
\hypersetup{pdfstartview={FitH}}

\newtheorem{thm}{Theorem}[section]
\newtheorem{lem}[thm]{Lemma}
\newtheorem{prop}[thm]{Proposition}
\newtheorem{cor}[thm]{Corollary}
\newtheorem{prob}[thm]{Problem}
\newtheorem{ques}[thm]{Question}
\newtheorem{conj}[thm]{Conjecture}
\theoremstyle{definition}
\newtheorem{eg}[thm]{Example}
\newtheorem{rem}[thm]{Remark}

\numberwithin{equation}{section}

\newcommand{\Mn}{\mathbb{M}_n}
\newcommand{\C}{\mathbb{C}}
\newcommand{\R}{\mathbb{R}}
\newcommand{\Id}{I}
\newcommand{\opnorm}[1]{\left\lVert #1 \right\rVert_{\infty}}
\newcommand{\hsnorm}[1]{\left\lVert #1 \right\rVert_{2}}
\newcommand{\tr}{\operatorname{Tr}}
\newcommand{\diag}{\operatorname{diag}}

\newcommand{\absv}[1]{\left|#1\right|}
\newcommand{\syma}[1]{\left|#1\right|_{\mathrm{sym}}}
\newcommand{\qsma}[1]{\left|#1\right|_{\mathrm{qsym}}}

\allowdisplaybreaks[4]
\begin{document}
	
	\title[symmetric moduli and triangle inequalities]{Operator symmetric moduli and sharp triangle inequalities}
	
	\author[T.~Zhang]{Teng Zhang}
	\address{School of Mathematics and Statistics, Xi'an Jiaotong University, Xi'an 710049, P. R. China}
	\email{teng.zhang@stu.xjtu.edu.cn}
	

\begin{abstract}
	We compare the usual operator modulus with two symmetrized variants, the  arithmetic symmetric modulus and the quadratic symmetric modulus. For every unitarily invariant norm, we determine sharp equivalence constants among these three moduli. We also establish sharp triangle-type inequalities for unitarily invariant norms, controlling sums of matrices by sums of symmetrized moduli, including optimal Schatten $p$-norm bounds and a phase transition phenomenon for the quadratic version. Explicit low-dimensional examples are provided to show that the constants are best possible. In particular, we answer two questions posed by Bourin and Lee in \cite{BL26b}.
\end{abstract}

	\maketitle
	\tableofcontents
	
	\section{Introduction}
	
	\subsection{Operator moduli and their equivalence}
	
	For a complex number $z=\Re z+i\,\Im z\in\C$, the modulus is \emph{classically} defined by
	\begin{equation}\label{eq:z-def-1}
		|z|:=\sqrt{(\Re z)^2+(\Im z)^2},
	\end{equation}
	which immediately implies
	\begin{equation}\label{eq:z-def-2}
		|z|=(\bar z z)^\frac{1}{2}
	\end{equation}
	and
	\begin{equation}\label{eq:z-def-3}
		|z|=\frac{|z|+|\bar{z}|}{2}.
	\end{equation}
	
Let $\mathbb{M}_{n_1,n_2}$ denote the set of all $n_1\times n_2$ complex matrices. When $n_1=n_2=n$, we write $\mathbb{M}_n$. For $Z\in\mathbb{M}_n$, the \emph{usual modulus}, in analogy with \eqref{eq:z-def-2}, is defined by
\[
|Z| := (Z^*Z)^{\frac{1}{2}},
\]
where $Z^*$ denotes the conjugate transpose of $Z$.

On the other hand, writing the Cartesian decomposition $Z=\Re Z+i\,\Im Z$, where
\[
\Re Z := \frac{Z+Z^*}{2},\qquad \Im Z := \frac{Z-Z^*}{2i},
\]
a more ``geometric” analogue of \eqref{eq:z-def-1} motivates the definition of the \emph{quadratic symmetric modulus}:
\[
|Z|_{\mathrm{qsym}} := \sqrt{(\Re Z)^2+(\Im Z)^2}
= \left(\frac{|Z|^2+|Z^*|^2}{2}\right)^{\frac{1}{2}}.
\]
We also define the \emph{arithmetic symmetric modulus}, following \eqref{eq:z-def-3}, by
\[
|Z|_{\mathrm{sym}} := \frac{|Z|+|Z^*|}{2}.
\]
	
	A norm $\|\cdot\|$ on $\Mn$ is called \emph{unitarily invariant} (or \emph{symmetric}) if
	\[
	\|UZV\|=\|Z\|\quad\text{for all unitaries }U,V.
	\]
	Our first objective is to characterize the equivalence among the three moduli introduced above, that is, to determine the \emph{sharp constants} $C_1,C_2$ such that
	\[
	C_1\,\bigl\|\,|Z|_1\,\bigr\|\ \le\ \bigl\|\,|Z|_2\,\bigr\|\ \le\ C_2\,\bigl\|\,|Z|_1\,\bigr\|
	\quad\text{for all } Z\in\Mn,
	\]
	where $|\cdot|_1,|\cdot|_2\in\{\,|\cdot|,\ |\cdot|_{\mathrm{sym}},\ |\cdot|_{\mathrm{qsym}}\,\}$ are two distinct choices of modulus.
	
	We state our main results as follows.	
	
	\begin{thm}[Usual vs arithmetic symmetric modulus]\label{thm:abs-vs-sym}
		For every unitarily invariant norm $\|\cdot\|$ on $\Mn$ and every $Z\in\Mn$,
		\[
		\frac12\,\bigl\||Z|\bigr\|\ \le\ \bigl\||Z|_{\mathrm{sym}}\bigr\|\ \le\ \bigl\||Z|\bigr\|.
		\]
		Moreover, both constants $\tfrac12$ and $1$ are sharp.
	\end{thm}
	
	\begin{thm}[Usual vs quadratic symmetric modulus]\label{thm:abs-vs-qsym}
		For every unitarily invariant norm $\|\cdot\|$ on $\Mn$ and every $Z\in\Mn$,
		\[
		\frac{\sqrt2}2\,\bigl\||Z|\bigr\|\ \le\ \bigl\||Z|_{\mathrm{qsym}}\bigr\|\ \le\ \sqrt2\,\bigl\||Z|\bigr\|.
		\]
		Moreover, both constants $\frac{\sqrt{2}}{2}$ and $\sqrt2$ are sharp.
	\end{thm}

	\begin{thm}[Arithmetic vs quadratic symmetric modulus]\label{thm:sym-vs-qsym}
		For every unitarily invariant norm $\|\cdot\|$ on $\Mn$ and every $Z\in\Mn$,
		\[
		\bigl\||Z|_{\mathrm{sym}}\bigr\|\ \le\ \bigl\||Z|_{\mathrm{qsym}}\bigr\|\ \le\ \sqrt2\,\bigl\||Z|_{\mathrm{sym}}\bigr\|.
		\]
		Moreover, both constants $1$ and $\sqrt2$ are sharp.
	\end{thm}
	
	\subsection{Operator triangle inequality}\label{subsec:operator-triangle}
	
	As an operator analogue of the scalar triangle inequality
	\[
	|z_1+z_2|\le |z_1|+|z_2|\qquad (z_1,z_2\in\C),
	\]
	the celebrated Thompson's inequality~\cite{Tho76} asserts the following.
	
	\begin{thm}[Thompson]\label{thm:thompson}
		Let $X,Y\in\Mn$. Then there exist unitaries $U,V$ such that
		\[
		|X+Y|\le U|X|U^*+V|Y|V^*.
		\]
	\end{thm}
	
	For the quadratic symmetric modulus, an exciting Thompson-type inequality in the matrix setting was recently established by the author in \cite[Theorem~1.12(ii)]{Zha26}, answering a question raised by Bourin and Lee \cite{BL26}. We also note that an alternative proof has since appeared in \cite{BL26b}.
	
	\begin{thm}[Zhang]\label{thm:qsym-thompson-matrix}
		Let $X,Y\in\Mn$. Then there exist unitaries $U,V$ such that
		\[
		|X+Y|_{\mathrm{qsym}}\le U|X|_{\mathrm{qsym}}U^*+V|Y|_{\mathrm{qsym}}V^*.
		\]
	\end{thm}
	In contrast, for the arithmetic symmetric modulus the situation is more subtle. In \cite{Zha26} the author exhibited a $2\times2$ counterexample showing that the triangle inequality may fail even for the operator norm: there exist $A,B\in\mathbb M_2$ such that
	\[
	\bigl\||A+B|_{\mathrm{sym}}\bigr\|_\infty \;>\; \bigl\||A|_{\mathrm{sym}}\bigr\|_\infty \;+\; \bigl\||B|_{\mathrm{sym}}\bigr\|_\infty.
	\]
	Consequently, a Thompson-type domination for $|\cdot|_{\mathrm{sym}}$ cannot hold in full generality; that is, one cannot in general find unitaries $U,V$ such that
	\[
	|X+Y|_{\mathrm{sym}}\le U|X|_{\mathrm{sym}}U^*+V|Y|_{\mathrm{sym}}V^*.
	\]
	Nevertheless, we conjecture that such a domination becomes valid after allowing a universal constant.
	
	\begin{conj}\label{conj:triangle-ineq-sym-modulus}
		Let $X,Y\in\Mn$. Then there exist unitaries $U,V$ such that
		\[
		|X+Y|_{\mathrm{sym}}\le \sqrt{2}\left(  U|X|_{\mathrm{sym}}U^*+V|Y|_{\mathrm{sym}}V^*\right),
		\]
		and the constant $\sqrt{2}$ is sharp.
	\end{conj}
 Conjecture~\ref{conj:triangle-ineq-sym-modulus} naturally suggests a weaker question: does there exist a universal constant $C>0$ such that for every $X,Y\in\mathbb M_n$ one can find unitaries $U,V$ satisfying
\begin{equation}\label{eq:weak-triangle-sym}
	|X+Y|_{\mathrm{sym}}\le C\left(U|X|_{\mathrm{sym}}U^*+V|Y|_{\mathrm{sym}}V^*\right)?
\end{equation}
Even this weaker formulation appears to be quite delicate.

A seemingly natural approach is to exploit the operator concavity of $t\mapsto t^{\frac{1}{2}}$, which yields
\[
|X+Y|_{\mathrm{sym}}\le |X+Y|_{\mathrm{qsym}}.
\]
Combining this with the Thompson-type inequality for the quadratic symmetric modulus (Theorem~\ref{thm:qsym-thompson-matrix}), we obtain that there exist unitaries $U,V$ such that
\begin{equation}\label{eq:sym-via-qsym}
	|X+Y|_{\mathrm{sym}}\le U|X|_{\mathrm{qsym}}U^*+V|Y|_{\mathrm{qsym}}V^*.
\end{equation}
Thus, to deduce \eqref{eq:weak-triangle-sym} from \eqref{eq:sym-via-qsym}, it would suffice to control $|X|_{\mathrm{qsym}}$ in terms of $|X|_{\mathrm{sym}}$ up to unitary conjugation. More precisely, one would need a universal constant $C>0$ such that for every $X\in\mathbb M_n$ there exists a unitary $W$ with
\begin{equation}\label{eq:qsym-dom-sym-unitary}
	|X|_{\mathrm{qsym}}\le C\,W|X|_{\mathrm{sym}}W^*.
\end{equation}
Equivalently, by Weyl's monotonicity, \eqref{eq:qsym-dom-sym-unitary} would imply the eigenvalue-wise bounds
\begin{equation}\label{eq:eigwise-qsym-sym}
	\lambda_j\!\left(|X|_{\mathrm{qsym}}\right)\le C\,\lambda_j\!\left(|X|_{\mathrm{sym}}\right),
	\qquad j=1,\dots,n,
\end{equation}
where $\lambda_j(\cdot)$ denotes the $j$th largest eigenvalue.

However, we will show by constructing an explicit $2\times2$ counterexample that no such universal constant $C$ exists in \eqref{eq:eigwise-qsym-sym}; see Proposition~\ref{prop:no-constant}, hence the strategy above cannot yield \eqref{eq:weak-triangle-sym}.
	\subsection{Triangle inequalities for unitarily invariant norms}
	Throughout this paper, we always assume $m\ge 2, n\ge 2$.
	An interesting  result of Lee \cite{Lee12} gives the following bound for sums of matrices under any unitarily invariant norm,
	\begin{equation}\label{eq:Lee-uin}
		\Bigl\|\sum_{k=1}^m A_k\Bigr\|\ \le\ \sqrt{m}\,\Bigl\|\sum_{k=1}^m |A_k|\Bigr\|, \qquad A_1,\ldots,A_m\in\Mn.
	\end{equation}
However,  sharpness of the constant $\sqrt{m}$ in \eqref{eq:Lee-uin} is unknown. In this paper, we show that the constant $\sqrt{m}$ in \eqref{eq:Lee-uin} is sharp only when $n \ge m$. Moreover, we derive the following sharp bound, which refines \eqref{eq:Lee-uin}.
	\begin{thm}\label{thm:Lee-optimal}
		Let $A_1,\dots,A_m\in\Mn$. Then
		\begin{equation*}
			\Bigl\|\sum_{k=1}^m A_k\Bigr\|\ \le\ \sqrt{\min \{m,n\}}\,\Bigl\|\sum_{k=1}^m |A_k|\Bigr\|.
		\end{equation*}
		Moreover, the constant $\sqrt{\min\{m,n\}}$ is sharp.
	\end{thm}

We also establish a direct sharp comparison between the sum $\sum_k A_k$ and the sum of the symmetric moduli.
	\begin{thm}\label{thm:sum-vs-sym}
		Let $A_1,\dots,A_m\in\Mn$. Then
		\[
		\Bigl\lVert \sum_{k=1}^m A_k \Bigr\rVert
		\ \le\ 2\,
		\Bigl\lVert \sum_{k=1}^m \bigl|A_k\bigr|_{\mathrm{sym}} \Bigr\rVert .
		\]
		Moreover, the constant $2$ is sharp.
	\end{thm}
	
		Turning to the quadratic symmetric modulus, we next record the analogue of Theorems~\ref{thm:Lee-optimal} and \ref{thm:sum-vs-sym},
	showing that $\sum_k A_k$ can be dominated in norm by $\sum_k \qsma{A_k}$ with the sharp constant $\sqrt2$.
	
	\begin{thm}\label{thm:sum-vs-qsym}
		Let $A_1,\dots,A_m\in\Mn$. Then
		\[
		\Bigl\lVert \sum_{k=1}^m A_k \Bigr\rVert
		\ \le\ \sqrt2\,
		\Bigl\lVert \sum_{k=1}^m \bigl|A_k\bigr|_{\mathrm{qsym}} \Bigr\rVert .
		\]
		Moreover, the constant $\sqrt2$ is sharp.
	\end{thm}
	
Since every unitarily invariant norm depends only on singular values, we have
\[
\Bigl\|\sum_{k=1}^m A_k\Bigr\|
=\Bigl\|\Bigl|\sum_{k=1}^m A_k\Bigr|\Bigr\|.
\]
This observation suggests seeking Lee-type estimates in which the modulus $|\cdot|$ on both sides of \eqref{eq:Lee-uin} is replaced by the symmetric moduli $\syma{\cdot}$ and $\qsma{\cdot}$ introduced above.

Recently, Bourin and Lee \cite{BL26b} established a genuine triangle-type inequality
	for the arithmetic symmetric modulus itself, namely an estimate for $\syma{\sum_k A_k}$ in terms of $\sum_k \syma{A_k}$,
	and \cite[Question~2.6]{BL26b} asked  whether the factor $\sqrt2$ in Theorem~\ref{thm:Bourin-Lee} could be improved. In this paper, we show that $\sqrt2$ is optimal for all $n\ge 3$,
	already for the operator norm.
	
	\begin{thm}[Bourin--Lee]\label{thm:Bourin-Lee}
		Let $A_1,\dots,A_m\in\Mn$. Then
		\begin{equation*}
			\Bigl\lVert \Bigl|\sum_{k=1}^m A_k\Bigr|_{\mathrm{sym}} \Bigr\rVert
			\ \le\ \sqrt2\,
			\Bigl\lVert \sum_{k=1}^m \bigl|A_k\bigr|_{\mathrm{sym}} \Bigr\rVert .
		\end{equation*}
		Moreover, the constant $\sqrt2$ is sharp for all $n\ge 3$.
	\end{thm}
However, for $2\times2$ matrices, numerical examples suggest that the optimal constant may be strictly smaller than $\sqrt{2}$ in Theorem~\ref{thm:Bourin-Lee}. This motivates the following question, which may be quite challenging.
	\begin{prob}\label{prob:Bourin-Lee-2-by-2}
			Let $A_1,\dots,A_m\in\mathbb{M}_2$. Determine the optimal constant $c(m,2)$ such that
		\begin{equation*}
			\Bigl\lVert \Bigl|\sum_{k=1}^m A_k\Bigr|_{\mathrm{sym}} \Bigr\rVert
			\ \le\ c(m,2)\,
			\Bigl\lVert \sum_{k=1}^m \bigl|A_k\bigr|_{\mathrm{sym}} \Bigr\rVert .
		\end{equation*}
	\end{prob}
Next, we give a $2\times2$ example that yields a nontrivial lower bound for Problem~\ref{prob:Bourin-Lee-2-by-2} in the operator norm.
	\begin{eg}
		\label{ex:n=2-lowerbound-1.1789}
		Let
\begin{align*}
		A&=
	\begin{pmatrix}
		-0.773354-3.706913\,i & -0.605203-0.251180\,i\\
		\ \ 0.302923+1.869626\,i & \ \ 0.296552+0.139226\,i
	\end{pmatrix},\\
	B&=
	\begin{pmatrix}
		-0.614194+0.304837\,i & -0.919027+0.530163\,i\\
		\ \ 2.687653+0.304749\,i & \ \ 4.176505+0.211368\,i
	\end{pmatrix}.
\end{align*}
	
Then a direct numerical evaluation (e.g.\ in \textsc{Matlab}) shows that, for the operator norm,
\[
\frac{\bigl\||A+B|_{\mathrm{sym}}\bigr\|_\infty}
{\bigl\||A|_{\mathrm{sym}}+|B|_{\mathrm{sym}}\bigr\|_\infty}
\approx 1.1789471123.
\]
Consequently, for every $m\ge 2$ (by taking $A_1=A$, $A_2=B$, and $A_3=\cdots=A_m=0$),
the optimal constant $c(m,2)$ in Problem~\ref{prob:Bourin-Lee-2-by-2} satisfies
\[
c(m,2)\ \ge\ 1.178.
\]
Moreover, extensive numerical experiments suggest that the optimal value of $c(m,2)$ might be close to this lower bound.
	\end{eg}
	Finally, we turn to Lee-type estimates in which the quadratic symmetric modulus
	appears on both sides. In contrast to the arithmetic symmetric modulus setting above, one can obtain
	a dimension-sensitive bound whose  constant depends on $\min\{m, 2n\}$.
	\begin{thm}\label{thm:qsym-Lee}
		Let $A_1,\dots,A_m\in\Mn$. Then
		\[
		\Bigl\lVert \Bigl|\sum_{k=1}^m A_k\Bigr|_{\mathrm{qsym}} \Bigr\rVert
		\ \le\ \sqrt{\min \{m,2n\}}\,
		\Bigl\lVert \sum_{k=1}^m \bigl|A_k\bigr|_{\mathrm{qsym}} \Bigr\rVert .
		\]
	\end{thm}
Furthermore, we strongly suspect that the constant in Theorem~\ref{thm:qsym-Lee} is not optimal; in particular, we have not been able to find any examples that attain  equality. This leaves an open problem.
	
	\subsection{Triangle inequalities on Schatten $p$-norms} Let $1 \le p \le \infty$. For $X \in \mathbb{M}_n$, the Schatten $p$-norm of $X$ is defined by
	\[
	\|X\|_p := \bigl(\tr|X|^p\bigr)^{\frac{1}{p}}.
	\]

Lee \cite{Lee12} posed the problem of determining the optimal constant $c_p$ such that
\[
\|A+B\|_p\le c_p\,\||A|+|B|\|_p,\qquad \text{for all } A,B\in \Mn,
\]
which is now commonly referred to as \emph{Lee's problem}.
In particular, she conjectured that
$c_2=\sqrt{\frac{1+\sqrt{2}}{2}}$.
This conjecture was first confirmed by Lin and Zhang \cite{LZ22}.
Subsequently, the author \cite{Zha25} provided an alternative proof, and \cite{Zha25b} further refined the estimate by deriving a sharper bound in which the constant is expressed more precisely in terms of the singular values of $A$ and $B$.

Motivated by Lee's problem, Tang and Zhang \cite{TZ26} proposed the following generalization for several matrices $A_k\in \Mn, 1\le k\le m$:
determine the optimal constant $c_p(m)$ depending only on $p$ and $m$ such that
\begin{equation}\label{eq:tang-zhang}
	\Bigl\lVert  \sum_{k=1}^m A_k\Bigr\rVert_p
	\ \le\ c_p(m)\,
	\Bigl\lVert \sum_{k=1}^m \bigl|A_k\bigr| \Bigr\rVert_p .
\end{equation}
For the above problem, Tang and Zhang~\cite{TZ26} stated several sharp values and general bounds. More precisely, they asserted that
\begin{enumerate}
	\item $c_{1}(m)=1$;
	\item $c_{2}(m)=\sqrt{\frac{1+\sqrt{m}}{2}}$, and they exhibited $m$ matrices of size $m\times m$ for which equality is attained in \eqref{eq:tang-zhang};
	\item $c_{\infty}(m)=\sqrt{m}$, and they exhibited $m$ matrices of size $m\times m$ for which equality is attained;
	\item $c_{p}(m)\le m^{\frac12-\frac{1}{2p}}$ for all $1\le p\le\infty$. 
\end{enumerate}
However, in order to claim that $c_{2}(m)=\sqrt{\frac{1+\sqrt{m}}{2}}$ and $c_{\infty}(m)=\sqrt{m}$ hold \emph{dimension-independently}, one would need to produce extremal examples in every dimension $n\ge 2$, in particular using $m$ matrices in $\mathbb M_{2}$ (or at least in $\mathbb M_n$ for arbitrary $n$). The constructions in~\cite{TZ26} only show sharpness in dimensions $n\ge m$: indeed, by taking a direct sum with a zero block, their $m\times m$ extremizers embed into $\mathbb M_n$ for any $n\ge m$, and hence the claimed constants are sharp in that range.
Indeed, the constant $c_p(m)$ in \eqref{eq:tang-zhang} is often \emph{dimension-dependent}. In other words, one cannot generally expect the optimal constant in
\eqref{eq:tang-zhang}
to be chosen independently of the matrix size $n$; rather, the best constant may vary with $n$. This motivates introducing the dimension-sensitive constant $c_p(m,n)$ and studying the sharper problem stated above.
\begin{prob}\label{prob:general-lee-problem}
	Let $1\le p\le \infty$ and let $A_1,\dots,A_m\in\Mn$. Determine the optimal constant $c_p^{\mathrm{abs}}(m,n)$ such that
	\[
	\Bigl\lVert  \sum_{k=1}^m A_k\Bigr\rVert_p
	\ \le\ c_p^{\mathrm{abs}}(m,n)\,
	\Bigl\lVert \sum_{k=1}^m \bigl|A_k\bigr| \Bigr\rVert_p .
	\]
\end{prob}

For Problem~\ref{prob:general-lee-problem}, we prove the following dimension-dependent conclusions:
\begin{enumerate}
	\item $c_{1}^{\mathrm{abs}}(m,n)=1$;
	\item $c_{2}^{\mathrm{abs}}(m,n)=\sqrt{\frac{1+\sqrt{\min\{m,n\}}}{2}}$;
	\item $c_{\infty}^{\mathrm{abs}}(m,n)=\sqrt{\min\{m,n\}}$;
	\item $c_{p}^{\mathrm{abs}}(m,n)\le \bigl(\min\{m,n\}\bigr)^{\frac12-\frac{1}{2p}}$ for all $1\le p\le\infty$.
\end{enumerate}
In particular, these bounds are dimension-sensitive and, when $m>n$, they strictly sharpen the corresponding estimates in Tang and Zhang~\cite{TZ26}.

Bourin and Lee \cite[Section~6]{BL26b} asked
\begin{center}
	\emph{What are the symmetric moduli versions of \eqref{eq:tang-zhang}?}
\end{center}
This motivates us to  investigate analogous inequalities in which the usual modulus $|\cdot|$ is replaced by the arithmetic symmetric modulus $|\cdot|_{\mathrm{sym}}$ or  the quadratic symmetric modulus $|\cdot|_{\mathrm{qsym}}$.

The following result gives a sharp Schatten-$p$ estimate in terms of the arithmetic symmetric modulus $|\cdot|_{\mathrm{sym}}$.

\begin{thm}\label{thm:schatten-sum-vs-sym-p}
	Let $1\le p\le \infty$ and let $A_1,\dots,A_m\in\Mn$. Then
	\[
	\Bigl\lVert \sum_{k=1}^m A_k \Bigr\rVert_p
	\ \le 2^{1-\frac{1}{p}}\,
	\Bigl\lVert \sum_{k=1}^m \bigl|A_k\bigr|_{\mathrm{sym}} \Bigr\rVert_p.
	\]
	Moreover, the constant $2^{1-\frac{1}{p}}$ is sharp.
\end{thm}

We next turn to the quadratic symmetric modulus. In this case the behavior is markedly different:
the optimal constant is independent of $m$ and exhibits a phase transition at $p=2$.

\begin{thm}\label{thm:schatten-sum-vs-qsym-p}
	Let $1\le p\le \infty$ and let $A_1,\dots,A_m\in\Mn$. Then
\begin{align*}
	\Bigl\lVert \sum_{k=1}^m A_k \Bigr\rVert_p
	\ &\le\ \,
	\Bigl\lVert \sum_{k=1}^m \bigl|A_k\bigr|_{\mathrm{qsym}} \Bigr\rVert_p, \quad\quad\quad\,\, 1\le p\le 2;\\
		\Bigl\lVert \sum_{k=1}^m A_k \Bigr\rVert_p
	\ &\le\ 2^{\frac{1}{2}-\frac{1}{p}}\,
	\Bigl\lVert \sum_{k=1}^m \bigl|A_k\bigr|_{\mathrm{qsym}} \Bigr\rVert_p, \quad 2\le p\le \infty.
\end{align*}
Moreover, the constants $1$ and $ 2^{\frac{1}{2}-\frac{1}{p}}$ are sharp.
\end{thm}
Having established these ``sum versus sum'' inequalities, it is natural to seek genuine triangle-type estimates in which the arithmetic symmetric modulus (or its quadratic analogue) is applied to the total sum on the left-hand side and compared directly with the sum of the individual moduli on the right-hand side. This leads to the following two optimization problems.
\begin{prob}\label{prob:schatten-sym-vs-sym-p}
	Let $A_1,\dots,A_m\in\Mn$. Determine the optimal constant $c_p^{\mathrm{sym}}(m,n)$ such that
	\[
	\Bigl\lVert  \bigl|\sum_{k=1}^m A_k\bigr|_{\mathrm{sym}} \Bigr\rVert_p
	\ \le\ c_p^{\mathrm{sym}}(m,n)\,
	\Bigl\lVert \sum_{k=1}^m \bigl|A_k\bigr|_{\mathrm{sym}} \Bigr\rVert_p .
	\]
\end{prob}
In Problem~\ref{prob:schatten-sym-vs-sym-p}, we prove that
\begin{enumerate}
\item		$c_1^{\mathrm{sym}}(m,n)=1$;
\item  $	c_\infty^{\mathrm{sym}}(m,n)=\sqrt2$ for $n\ge 3$;
\item $c_p^{\mathrm{sym}}(m,n)\le \min\{2^{1-\frac{1}{p}}, \sqrt{2}\}$;
\item $	c_p^{\mathrm{sym}}(m,n)\ge \left(\frac{2^{\frac{p}{2}}+2^{1-\frac{p}{2}}}{3}\right)^{\frac{1}{p}}$;
in particular, the right-hand side increases from $1$ at $p=2$ to $\sqrt2$ as $p\to\infty$.
\end{enumerate}
\begin{prob}\label{prob:schatten-qsym-vs-qsym-p}
	Let $A_1,\dots,A_m\in\Mn$. Determine the optimal constant $c_p^{\mathrm{qsym}}(m,n)$ such that
	\[
	\Bigl\lVert \bigl|\sum_{k=1}^m A_k\bigr|_{\mathrm{qsym}} \Bigr\rVert_p
	\ \le\ c_p^{\mathrm{qsym}}(m,n)\,
	\Bigl\lVert \sum_{k=1}^m \bigl|A_k\bigr|_{\mathrm{qsym}} \Bigr\rVert_p .
	\]
\end{prob}
For Problem~\ref{prob:schatten-qsym-vs-qsym-p}, we reduce the question to
Problem~\ref{prob:general-lee-problem} by constructing an appropriate $2\times2$ block matrix, which yields
\begin{enumerate}
	\item $c_{1}^{\mathrm{qsym}}(m,n)=1$;
	\item $c_{2}^{\mathrm{qsym}}(m,n)\le \sqrt{\frac{1+\sqrt{\min\{m,2n\}}}{2}}$;
	\item $c_{\infty}^{\mathrm{qsym}}(m,n)\le \sqrt{\min\{m,2n\}}$;
	\item $c_{p}^{\mathrm{qsym}}(m,n)\le \bigl(\min\{m,2n\}\bigr)^{\frac12-\frac{1}{2p}}$ for all $1\le p\le\infty$.
\end{enumerate}
\medskip
\noindent\textbf{Proof techniques in this paper.}
The proofs in this paper rely on the following tools, combined with
explicit low-dimensional extremizers for sharpness.

\smallskip
\noindent
\emph{(i) Functional calculus, operator concavity/monotonicity, and subadditivity.}
The equivalence constants among the three moduli are obtained by applying the operator concavity and
monotonicity of the square-root map, together with standard norm inequalities for unitarily invariant norms.
More precisely, Theorems~\ref{thm:abs-vs-sym}, \ref{thm:abs-vs-qsym}, and \ref{thm:sym-vs-qsym}
follow from elementary order comparisons (e.g.\ $|Z|\le |Z|+|Z^*|$), the triangle inequality,
and concavity-based inequalities such as the Bourin--Uchiyama subadditivity for symmetric norms
applied to $f(t)=t^\frac{1}{2}$.
Sharpness in these theorems is certified by concrete $2\times2$ examples (notably $E_{12}$) and by normal matrices.

\smallskip
\noindent
\emph{(ii) Positive block matrices, contractions, and Ky Fan domination.}
Several triangle-type bounds are proved by encoding the summands into a positive $2\times2$ block matrix and then
extracting norm inequalities from this positivity.
The dimension-dependent refinement of Lee's inequality in Theorem~\ref{thm:Lee-optimal} is derived by:
(1) summing the standard positive block matrix
$\bigl(\begin{smallmatrix}|A_k^*|&A_k\\ A_k^*&|A_k|\end{smallmatrix}\bigr)\ge 0$,
(2) factoring the off-diagonal block via a contraction,
and (3) estimating Ky Fan norms and invoking Fan's dominance theorem.
The ``sum versus sum'' inequalities
(Theorems~\ref{thm:sum-vs-sym} and \ref{thm:sum-vs-qsym})
use an equal-diagonal block-positivity criterion to bound the off-diagonal block by the diagonal block for every
unitarily invariant norm.

\smallskip
\noindent
\emph{(iii) Embedding tricks for the quadratic symmetric modulus.}
To transfer estimates for the usual modulus to the quadratic symmetric modulus, we introduce a block embedding
$\Phi(A)$ whose modulus satisfies
$|\Phi(A)|=\diag(|A|_{\mathrm{qsym}},0)$.
This reduction allows us to deduce Theorem~\ref{thm:qsym-Lee} from Theorem~\ref{thm:Lee-optimal}
by working in dimension $2n$.

\smallskip
\noindent
\emph{(iv) Schatten $p$-norm arguments: dilations, duality, and a phase transition.}
For Schatten norms, Theorem~\ref{thm:schatten-sum-vs-sym-p} uses Hermitian dilations
$H_k=\bigl(\begin{smallmatrix}0&A_k\\ A_k^*&0\end{smallmatrix}\bigr)$, block positivity, contraction factorization,
Schatten--H\"{o}lder inequalities, and the McCarthy trace inequality to compare $\|\sum H_k\|_p$ and $\|\,\sum|H_k|\,\|_p$.
Theorem~\ref{thm:schatten-sum-vs-qsym-p} proceeds by Schatten duality and a trace domination of the form
$$|\tr(A^*B)|\le \tr(|A|_{\mathrm{qsym}}|B|_{\mathrm{qsym}})$$ (deduced by Lieb's concavity theorem), which reduces the problem to a single-matrix constant.
Computing this constant uses convexity/concavity of $t\mapsto t^{\frac{q}{2}}$ and leads to a sharp phase transition at $p=2$,
with extremizers again given by $E_{12}$.

\smallskip
\noindent
\emph{(v) Sharpness and obstruction examples.}
Sharpness of the constants is established by explicit extremal constructions:
rank-one examples in $\mathbb M_3$ (Theorem~\ref{thm:sharp}) for the Bourin--Lee constant,
and $2\times2$ families with degenerating spectra to rule out any universal unitary domination
(Proposition~\ref{prop:no-constant}).

\medskip
\noindent\textbf{Organization of this paper.}
In Section~\ref{sec:bourin-lee-problem} we address two questions of Bourin and Lee:
we prove optimality of the factor $\sqrt2$ for all $n\ge 3$ by an explicit $3\times3$ construction,
and we provide counterexamples relevant to the expansive decomposition problem.
Section~\ref{sec:equivalence} contains the proofs of the equivalence results for the three moduli
(Theorems~\ref{thm:abs-vs-sym}--\ref{thm:sym-vs-qsym})
and the nonexistence of a universal unitary domination constant (Proposition~\ref{prop:no-constant}).
Section~\ref{sec:sharpness} proves the sharp dimension-dependent refinement of Lee's inequality (Theorem~\ref{thm:Lee-optimal})
and clarifies when the original $\sqrt{m}$ factor is optimal.
Section~\ref{sec:sum-vs} establishes the sharp ``sum versus sum'' comparisons involving the symmetric and quadratic symmetric moduli
(Theorems~\ref{thm:sum-vs-sym}, \ref{thm:sum-vs-qsym}, and \ref{thm:qsym-Lee}).
Section~\ref{sec:partial-results-general-lee} develops partial results on the dimension-sensitive version of Lee's problem, including the sharp Frobenius-norm
constant and general $p$-bounds.
Section~\ref{sec:proofs-schatten-p} proves the sharp Schatten $p$-norm inequalities in terms of the arithmetic symmetric modulus and the quadratic symmetric modulus
(Theorems~\ref{thm:schatten-sum-vs-sym-p} and \ref{thm:schatten-sum-vs-qsym-p}).
Finally, Section~\ref{sec:partial-schatten-p} collects partial results and bounds for the genuine triangle-type Schatten $p$-norm problems
involving $|\cdot|_{\mathrm{sym}}$ and $|\cdot|_{\mathrm{qsym}}$.

	\section{Solutions on two questions of Bourin and Lee}\label{sec:bourin-lee-problem}
	
	\subsection{Two questions of  Bourin and Lee}
Bourin and Lee~\cite{BL26b} proved Theorem~\ref{thm:Bourin-Lee} and also obtained the following eigenvalue bound.

\begin{cor}[{\cite[Corollary~2.4]{BL26b}}]\label{cor:BL-2.4}
	Let $A_1,\dots,A_m\in \mathbb M_n$. Then for all integers $j\ge 0$ such that $1+3j\le n$,
	\[
	\lambda_{1+3j}\!\left(\left|\sum_{k=1}^m A_k\right|_{\mathrm{sym}}\right)
	\ \le\
	\sqrt{2}\,
	\lambda_{1+j}\!\left(\sum_{k=1}^m |A_k|_{\mathrm{sym}}\right),
	\]
	where $\lambda_j(\cdot)$ denotes the $j$th largest eigenvalue.
\end{cor}

They further asked the following question.

\begin{ques}[{\cite[Question~2.6]{BL26b}}]\label{ques:BL-2.6}
	Is $\sqrt{2}$ the best possible constant in Theorem~\ref{thm:Bourin-Lee} and Corollary~\ref{cor:BL-2.4}?
\end{ques}
In this section, we answer Question~\ref{ques:BL-2.6} affirmatively for all $n\ge 3$.

Bourin and Lee also observed (see \cite[Section~6]{BL26b}) that if $A_1,\dots,A_m\in \mathbb{M}_7$
are such that $\sum_{k=1}^m A_k$ is a Hermitian unitary, then
\[
\lambda_4\!\left(\sum_{k=1}^m |A_k|_{\mathrm{sym}}\right)\ \ge\ 1.
\]
Thus, decomposing a Hermitian unitary as a sum imposes rather strange constraints on the summands,
which motivates the following question. Recall that $X\in \mathbb M_n$ is \emph{expansive} if $|X|\ge I$.

\begin{ques}[{\cite[Question~6.3]{BL26b}}]\label{ques:BL-expansive}
	Let $A\in \mathbb{M}_{2n}$ or $\mathbb{M}_{2n-1}$ be expansive, and suppose that $A=\sum_{k=1}^m A_k$.
	What can be said about
	\[
	\lambda_n\!\left(\sum_{k=1}^m |A_k|_{\mathrm{sym}}\right)
	\qquad\text{and}\qquad
	\lambda_n\!\left(\sum_{k=1}^m |A_k|\right)?
	\]
\end{ques}
We will show that, for the usual modulus term
\(\lambda_n\!\left(\sum_k |A_k|\right)\), a universal lower bound equal to \(1\)
cannot be expected, which yields a partial answer to
Question~\ref{ques:BL-expansive}.
	
	\subsection{Solution of Question~\ref{ques:BL-2.6} for all $n\ge 3$}
	We begin with an elementary rank-one computation.
	
	\begin{lem}\label{lem:rank1}
		Let $u,v\in \C^n$ be unit vectors and set $R:=u v^*\in \Mn$.
		Then
		\[
		R^* R = v v^*,\qquad RR^* = u u^*,\qquad
		\absv{R} = v v^*,\qquad \absv{R^*}=u u^*.
		\]
		Consequently $\syma{R} = \tfrac12(uu^* + vv^*)$.
	\end{lem}
	\begin{proof}
		By a direct computation.
	\end{proof}
	Next, we illustrate an explicit extremal pair in $\mathbb{M}_3$ in Theorem~\ref{thm:Bourin-Lee}.
	
	\begin{thm}\label{thm:sharp}
		The constant $\sqrt2$ in Theorem~\ref{thm:Bourin-Lee} is optimal for all $n\ge 3$.
		In fact, there exist $A,B\in \mathbb{M}_3$ such that
		\[
		\opnorm{\syma{A+B}}=\sqrt2\,\opnorm{\syma{A}+\syma{B}}.
		\]
		Moreover, for every $n> 3$, setting
		\[
		A_n:=A\oplus 0_{n-3},\quad B_n:=B\oplus 0_{n-3}\quad\in\Mn,
		\]
		one has
		\[
		\opnorm{\syma{A_n+B_n}}=\sqrt2\,\opnorm{\syma{A_n}+\syma{B_n}}.
		\]
	\end{thm}
	\begin{proof}
		We	work over $\R^3$, consider the unit vectors
		\[
		u=\frac1{\sqrt3}\!\begin{pmatrix}1\\1\\1\end{pmatrix},\quad
		v=\frac1{\sqrt3}\!\begin{pmatrix}1\\-1\\-1\end{pmatrix},\quad
		x=\frac1{\sqrt3}\!\begin{pmatrix}-1\\1\\-1\end{pmatrix},\quad
		y=\frac1{\sqrt3}\!\begin{pmatrix}1\\1\\-1\end{pmatrix}.
		\]
		Define $A:=u v^{\top}$ and $B:=x y^{\top}$.
		By Lemma~\ref{lem:rank1},
		\[
		\syma{A}=\frac12(uu^{\top}+vv^{\top}),\qquad
		\syma{B}=\frac12(xx^{\top}+yy^{\top}).
		\]
		A direct computation of outer products (each entry is $\pm 1/3$) shows that the off-diagonal
		terms cancel and
		\[
		uu^{\top}+vv^{\top}+xx^{\top}+yy^{\top}=\frac{4}{3}\,\Id_3.
		\]
		Hence
		\begin{equation}\label{eq:sum-symA-symB}
			\syma{A}+\syma{B}=\frac12\cdot \frac{4}{3}\,\Id_3=\frac{2}{3}\,\Id_3,
			\qquad\text{so}\qquad
			\opnorm{\syma{A}+\syma{B}}=\frac{2}{3}.
		\end{equation}
		
		Next, we compute $S:=A+B$ explicitly:
		\[
		S = A+B = \frac13
		\begin{pmatrix}
			0&-2&0\\
			2&0&-2\\
			0&-2&0
		\end{pmatrix}.
		\]
		One checks
		\[
		S^{\top}S=\frac19
		\begin{pmatrix}
			4&0&-4\\
			0&8&0\\
			-4&0&4
		\end{pmatrix},
		\qquad
		SS^{\top}=\frac19
		\begin{pmatrix}
			4&0&4\\
			0&8&0\\
			4&0&4
		\end{pmatrix}.
		\]
		Both matrices have spectrum $\{8/9,\,8/9,\,0\}$.
		Since $S^{\top}S$ is a scalar multiple of the orthogonal projection onto its range,
		\[
		\absv{S}=(S^{\top}S)^\frac{1}{2}=\frac{1}{\sqrt{8/9}}\,(S^{\top}S)=\frac{3\sqrt2}{4}\,S^{\top}S,
		\]
		and similarly $\absv{S^{\top}}=\frac{3\sqrt2}{4}\,SS^{\top}$.
		Therefore
		\[
		\syma{S}=\frac{\absv{S}+\absv{S^{\top}}}{2}
		=
		\frac{3\sqrt2}{8}\,(S^{\top}S+SS^{\top})
		=
		\diag\!\left(\frac{\sqrt2}{3},\,\frac{2\sqrt2}{3},\,\frac{\sqrt2}{3}\right),
		\]
		so
		\begin{equation}\label{eq:norm-symS}
			\opnorm{\syma{A+B}}=\opnorm{\syma{S}}=\frac{2\sqrt2}{3}.
		\end{equation}
		Combining \eqref{eq:sum-symA-symB} and \eqref{eq:norm-symS} yields
		\[
		\frac{\opnorm{\syma{A+B}}}{\opnorm{\syma{A}+\syma{B}}}
		=
		\frac{(2\sqrt2)/3}{2/3}=\sqrt2.
		\]
		Since the operator norm is a symmetric norm, no constant $c<\sqrt2$ can replace $\sqrt2$
		in Theorem~\ref{thm:Bourin-Lee} for all symmetric norms with $n\ge 3$.
		This proves optimality of $\sqrt2$.
	\end{proof}
	
	\begin{rem}
		Theorem~\ref{thm:sharp} already settles Question~\ref{ques:BL-2.6} for all $n\ge 3$, because Corollary~\ref{cor:BL-2.4}
		with $j=0$ is precisely the operator-norm estimate
		$\lambda_1(\syma{A+B})\le \sqrt2\,\lambda_1(\syma{A}+\syma{B})$.
	\end{rem}

	\subsection{Partial solution of Question~\ref{ques:BL-expansive}}
	\begin{prop}\label{prop:counterexample-modulus}
		Let $n=2$ and $X=\Id_3\in \mathbb{M}_{2n-1}$.
		There exist $X_1,X_2\in \mathbb{M}_3$ with $X=X_1+X_2$ such that
		\[
		\lambda_2\!\left(\absv{X_1}+\absv{X_2}\right) < 1.
		\]
		In particular, no general statement of the form
		$\lambda_n(\sum_k \absv{X_k})\ge 1$ can hold in Question~\ref{ques:BL-expansive}.
	\end{prop}
	
	\begin{proof}
		Set
		\[
		X_1 :=
		\begin{pmatrix}
			-1 & -1 & 0\\
			\phantom{-}1 & \phantom{-}1 & 0\\
			-1 & -1 & 1
		\end{pmatrix},
		\qquad
		X_2 := \Id_3 - X_1 =
		\begin{pmatrix}
			2 & 1 & 0\\
			-1 & 0 & 0\\
			1 & 1 & 0
		\end{pmatrix}.
		\]
		Then $X_1+X_2=\Id_3$, hence $X=\Id_3$ is expansive.
		A direct computation (e.g.\ via spectral decomposition of $X_k^* X_k$) shows that
		the eigenvalues of the positive semidefinite matrix $\absv{X_1}+\absv{X_2}$ are
		\[
		\lambda_1 \approx 5.11522680,\qquad
		\lambda_2 \approx 0.88353915,\qquad
		\lambda_3 \approx 0.70372677.
		\]
		Hence $\lambda_2(\absv{X_1}+\absv{X_2})\approx 0.8835<1$, as claimed.
	\end{proof}
	\section{Proofs of Theorems~\ref{thm:abs-vs-sym}--\ref{thm:sym-vs-qsym} and Proposition~\ref{prop:no-constant}}\label{sec:equivalence}
		\begin{proof}[Proof of Theorem~\ref{thm:abs-vs-sym}]
	\emph{A sharp lower bound.}	Since $|Z|\le |Z|+|Z^*|$, monotonicity implies $\||Z|\|\le  \||Z|+|Z^*|\|$, hence
		\[
		\frac12\||Z|\|\le \left\|\frac{|Z|+|Z^*|}{2}\right\|=\||Z|_{\mathrm{sym}}\|.
		\]
\medskip
\noindent\emph{Sharpness of the lower constant $\tfrac12$.}
Let $Z=E_{12}\in\mathbb M_2$. Then
$|Z|=\mathrm{diag}(0,1)$ and $|Z^*|=\mathrm{diag}(1,0)$, hence
$|Z|_{\mathrm{sym}}=\frac{|Z|+|Z^*|}{2}=\frac12 I_2$.
For the operator norm,
$\||Z|_{\mathrm{sym}}\|_\infty=\tfrac12$ while $\||Z|\|_\infty=1$,
so equality holds in $\frac12\||Z|\|\le \||Z|_{\mathrm{sym}}\|$.
		
		\medskip
\noindent	\emph{A sharp upper bound.} The triangle inequality and unitary invariance give
		\[
		\||Z|_{\mathrm{sym}}\|
		=\left\|\frac{|Z|+|Z^*|}{2}\right\|
		\le \frac{\|Z\|+\|Z^*\|}{2}
		=\||Z|\|.
		\]
Sharpness follows by taking  $Z$ to be a normal matrix.
	\end{proof}
Before proving Theorem~\ref{thm:sym-vs-qsym}, we recall the following subadditivity inequality due to Bourin and Uchiyama~\cite{BU07}.

	\begin{lem}[\cite{BU07}]\label{lem:Bourin-Uchiyama}
		Let $A,B\in \Mn$ be positive semidefinite and $f$ is a nonnegative concave function defined on $[0,\infty)$.  Then
		\[
		\|f(A+B)\|\le \|f(A)+f(B)\|.
		\]
		In particular, for the trace norm, we have
		\[
		\tr	f(A+B)\le  \tr \left(  f( A)+f(B) \right).
		\]
	\end{lem}
	\begin{proof}[Proof of Theorem~\ref{thm:sym-vs-qsym}]
		\emph{A sharp lower bound.}
		Since $t\mapsto t^\frac{1}{2}$ is operator concave on $[0,\infty)$,  we have
\begin{equation*}
		\frac{|Z|+|Z^*|}{2} \le \left(\frac{|Z|^2+|Z^*|^2}{2}\right)^\frac{1}{2},
\end{equation*}
		i.e.\ $|Z|_{\mathrm{sym}}\le |Z|_{\mathrm{qsym}}$. By monotonicity,
		$\||Z|_{\mathrm{sym}}\|\le \||Z|_{\mathrm{qsym}}\|$.
Sharpness can be obtained by  taking
	$Z$ to be normal.
		
		\medskip
	\noindent	\emph{A sharp upper bound.}
		We invoke the following known subadditivity inequality for symmetric norms obtained by Bourin--Uchiyama's Theorem (Lemma~\ref{lem:Bourin-Uchiyama}):
		if $X,Y\ge0$ and $f:[0,\infty)\to[0,\infty)$ is concave, then for every unitarily invariant norm,
		\[
		\|f(X+Y)\|\ \le\ \|f(X)+f(Y)\|.
		\]
		Applying it with $X=|Z|^2$, $Y=|Z^*|^2$ and $f(t)=\sqrt{t}$ gives
		\[
		\left\|(|Z|^2+|Z^*|^2)^\frac{1}{2}\right\|\le \||Z|+|Z^*|\|.
		\]
		Using homogeneity,
		\[
		\bigl\||Z|_{\mathrm{qsym}}\bigr\|
		=\frac{\sqrt2}2	\left\|(|Z|^2+|Z^*|^2)^\frac{1}{2}\right\|
		\le \frac{\sqrt2}2 \||Z|+|Z^*|\|
		=\sqrt2\,\bigl\||Z|_{\mathrm{sym}}\bigr\|.
		\]

		Take $Z=\begin{pmatrix}0&1\\0&0\end{pmatrix}$, 
		then $|Z|_{\mathrm{sym}}=\frac12 I_2$ and $|Z|_{\mathrm{qsym}}=\frac{\sqrt{2}}{2} I_2$.
	For the operator norm $\|\cdot\|_\infty$, we have
$\bigl\||Z|_{\mathrm{qsym}}\bigr\|_\infty=\sqrt2\,\bigl\||Z|_{\mathrm{sym}}\bigr\|_\infty$,
		showing $\sqrt2$ is sharp.
	\end{proof}
		\begin{proof}[Proof of Theorem~\ref{thm:abs-vs-qsym}]
		\emph{A sharp lower bound.}
		Since $\frac{|Z|^2}{2}\le \frac{|Z|^2+|Z^*|^2}{2}$, by the operator monotonicity of the function $t\mapsto t^\frac{1}{2}$, we have
		\[
		\frac{\sqrt2}2 |Z| \le \left(\frac{|Z|^2+|Z^*|^2}{2}\right)^\frac{1}{2}=|Z|_{\mathrm{qsym}}.
		\]
		This implies that
		$	\frac{\sqrt2}2\|Z\|\le \||Z|_{\mathrm{qsym}}\|$.
		
		For $Z=\begin{pmatrix}0&1\\0&0\end{pmatrix}$, then $|Z|_{\mathrm{qsym}}=\frac{\sqrt2}2I_2$ and $|Z|=\mathrm{diag}(0,1)$.
		For the operator norm $\|\cdot\|_\infty$,
		$
		\frac{\sqrt2}2\|Z\|_\infty= \||Z|_{\mathrm{qsym}}\|_\infty,
		$
		so the lower constant $	\frac{\sqrt2}2$ is sharp.
		
		\medskip
		\noindent\emph{A sharp upper bound.}
		Combine Theorem~\ref{thm:sym-vs-qsym} with Theorem~\ref{thm:abs-vs-sym}:
		\[
		\||Z|_{\mathrm{qsym}}\|\le \sqrt2\,\||Z|_{\mathrm{sym}}\|\le \sqrt2\,\||Z|\|.
		\]
		
		For $Z=\begin{pmatrix}0&1\\0&0\end{pmatrix}$, then $|Z|_{\mathrm{qsym}}=\frac{\sqrt2}2I_2$ and $|Z|=\mathrm{diag}(0,1)$.	For the trace norm $\|\cdot\|_1$,
		\[
		\||Z|\|_1=1,\qquad \||Z|_{\mathrm{qsym}}\|_1=\sqrt2,
		\]
		hence $\||Z|_{\mathrm{qsym}}\|_1=\sqrt2\||Z|\|_1$, showing the upper constant $\sqrt2$ is sharp.
	\end{proof}
We now verify the claim made at the end of Subsection~\ref{subsec:operator-triangle}.
\begin{prop}\label{prop:no-constant}
	There is no universal constant $c>0$ such that for every $X\in\mathbb M_n$
	there exists a unitary $U$ satisfying
	\[
	|X|_{\mathrm{qsym}}\le c\,U|X|_{\mathrm{sym}}U^*.
	\]
\end{prop}
\begin{proof} Set
$
X_\theta:=\begin{pmatrix}
	\cos\theta & 0\\
	\sin\theta & 0
\end{pmatrix}\in\mathbb{M}_2,
$ where $\theta\in(0,\pi/2)$.
Now we compute $|X_\theta|$ and $|X_\theta^*|$.
\[
X_\theta^*X_\theta=
\begin{pmatrix}1&0\\0&0\end{pmatrix}
=:P
\quad\Longrightarrow\quad
|X_\theta|=(X_\theta^*X_\theta)^\frac{1}{2}=P.
\]

\[
X_\theta X_\theta^*=
\begin{pmatrix}
	\cos^2\theta & \cos\theta\sin\theta\\
	\cos\theta\sin\theta & \sin^2\theta
\end{pmatrix}
=:Q_\theta,
\qquad Q_\theta^2=Q_\theta
\quad\Longrightarrow\quad
|X_\theta^*|=(X_\theta X_\theta^*)^\frac{1}{2}=Q_\theta.
\]
Next, we compute two symmetric moduli of $X$:
\[
|X_\theta|_{\rm sym}=\frac{P+Q_\theta}{2},
\qquad
|X_\theta|_{\rm qsym}=
\left(\frac{P^2+Q_\theta^2}{2}\right)^\frac{1}{2}
=
\left(\frac{P+Q_\theta}{2}\right)^\frac{1}{2}
=
\bigl(|X_\theta|_{\rm sym}\bigr)^\frac{1}{2}.
\]

Let $S_\theta:=P+Q_\theta$. Then
\[
S_\theta=
\begin{pmatrix}
	1+\cos^2\theta & \cos\theta\sin\theta\\
	\cos\theta\sin\theta & \sin^2\theta
\end{pmatrix},
\quad
\tr S_\theta=2,
\quad
\det S_\theta=\sin^2\theta,
\]
so the eigenvalues of $S_\theta$ are
\[
\lambda_{1,2}(S_\theta)=1\pm\cos\theta.
\]
Hence
\[
\lambda_{1,2}\bigl(|X_\theta|_{\rm sym}\bigr)=\frac{1\pm\cos\theta}{2},
\qquad
\lambda_{1,2}\bigl(|X_\theta|_{\rm qsym}\bigr)=\sqrt{\frac{1\pm\cos\theta}{2}}.
\]
Compute the ratio of the smallest eigenvalues of these two symmetric moduli,
\[
\frac{\lambda_2(|X_\theta|_{\rm qsym})}{\lambda_2(|X_\theta|_{\rm sym})}
=
\frac{\sqrt{\frac{1-\cos\theta}{2}}}{\frac{1-\cos\theta}{2}}
=
\sqrt{\frac{2}{1-\cos\theta}}
\longrightarrow +\infty
\quad (\theta\to 0^+).
\]
Using $1-\cos\theta<\theta^2/2$ for $\theta\neq 0$, we have
\[
\frac{\lambda_2(|X_\theta|_{\rm qsym})}{\lambda_2(|X_\theta|_{\rm sym})}
=
\sqrt{\frac{2}{1-\cos\theta}}
>
\sqrt{\frac{2}{\theta^2/2}}
=
\frac{2}{\theta},
\]
so for any $c>0$ choose $\theta<2/c$ to get
\[
\lambda_2(|X_\theta|_{\rm qsym})>c\,\lambda_2(|X_\theta|_{\rm sym}).
\]
Therefore no finite constant $c$ can satisfy
\[
\lambda_j(|X|_{\rm qsym})\le c\,\lambda_j(|X|_{\rm sym})
\quad\text{for all }X\in\mathbb{M}_2\text{ and all }j.
\]
\end{proof}

\section{Sharpness of the constant in \eqref{eq:Lee-uin}  and proof of Theorem~\ref{thm:Lee-optimal}}\label{sec:sharpness}
In this section, we first verify sharpness of the constant $\sqrt{m}$ in \eqref{eq:Lee-uin} for all $n\ge m$.
\subsection{Sharpness of the constant $\sqrt{m}$ in \eqref{eq:Lee-uin} for all $n\ge m$}
	\begin{eg}\label{eg:Lee-uin-sharp}
		Let $n\ge m$. For $k=1,\dots,m$, define
		\[
		A_k:=E_{1k}\in\mathbb M_n,
		\]
		where $E_{ij}$ denotes the matrix unit in $\mathbb M_n$.
		Then
		\[
		A_k^*A_k=E_{k1}E_{1k}=E_{kk},
		\qquad\text{so}\qquad
		|A_k|=(A_k^*A_k)^\frac{1}{2}=E_{kk}.
		\]
		Hence
		\[
		\sum_{k=1}^m |A_k|
		=\sum_{k=1}^m E_{kk},
		\qquad\Longrightarrow\qquad
		\Bigl\|\sum_{k=1}^m |A_k|\Bigr\|_\infty
		=\Bigl\|\mathrm{diag}(\underbrace{1,\dots,1}_{m},0,\dots,0)\Bigr\|_\infty=1.
		\]
		On the other hand,
		\[
		\sum_{k=1}^m A_k=\sum_{k=1}^m E_{1k},
		\]
		which is the $n\times n$ matrix whose first row is $(1,\dots,1,0,\dots,0)$ (with $m$ ones) and all other rows are zero. Therefore,
		\[
		\left(\sum_{k=1}^m A_k\right)\left(\sum_{k=1}^m A_k\right)^*
		=
		\left(\sum_{k=1}^m E_{1k}\right)\left(\sum_{k=1}^m E_{k1}\right)
		=
		\sum_{k=1}^m E_{11}
		=
		m\,E_{11},
		\]
		so the largest singular value of $\sum_{k=1}^m A_k$ equals $\sqrt m$, i.e.
		\[
		\Bigl\|\sum_{k=1}^m A_k\Bigr\|_\infty=\sqrt m.
		\]
		Consequently, equality holds in \eqref{eq:Lee-uin}:
		\[
		\Bigl\|\sum_{k=1}^m A_k\Bigr\|_\infty
		=
		\sqrt m\,
		\Bigl\|\sum_{k=1}^m |A_k|\Bigr\|_\infty,
		\]
		and hence the constant $\sqrt m$ in \eqref{eq:Lee-uin} is sharp for every $n\ge m$.
	\end{eg}
	
	\subsection{Proof of Theorem~\ref{thm:Lee-optimal}}
		The following result is standard in matrix analysis.
	\begin{lem}[{\cite[p.~13, Proposition~1.3.2]{Bha07}}]\label{lem:contraction}
		Let $A,B\in\Mn$ be positive semidefinite and let $X\in\Mn$. Then
		\[
		\begin{pmatrix}
			A & X\\
			X^* & B
		\end{pmatrix}\ge 0
		\quad\Longleftrightarrow\quad
		X=A^{\frac{1}{2}}KB^{\frac{1}{2}}\ \text{for some contraction }K.
		\]
	\end{lem}
	Next, we  recall a basic H\"{o}lder inequality for Schatten $p$-norms.
	In the special case $m=2$, it appears in {\cite[p.~95, Exercise IV.2.7]{Bha97}} by taking the trace norm $\|\cdot\|_1$.
	\begin{lem}\label{lem:Holder-inequality}
		Let $m\ge 2$, let $1\le r\le \infty$, and let $1\le p_1,\dots,p_m\le \infty$ satisfy
		\[
		\frac1r=\frac1{p_1}+\cdots+\frac1{p_m}.
		\]
		Then for all $A_1,\dots,A_m\in\mathbb M_n$,
		\[
		\bigl\|A_1A_2\cdots A_m\bigr\|_r
		\ \le\ 
		\|A_1\|_{p_1}\,\|A_2\|_{p_2}\cdots \|A_m\|_{p_m},
		\]
		where $\|\cdot\|_p$ denotes the Schatten $p$-norm.
	\end{lem}
	
We also need the following lemma.
	\begin{lem}\label{lem:pos-block-MA}
		Let $A\in\mathbb M_n$. Then
		\[
		\begin{pmatrix}
			|A^*| & A\\
			A^* & |A|
		\end{pmatrix}\ge 0.
		\]
	\end{lem}
	\begin{proof}
		Let $A=U|A|$ be a polar decomposition. Then $|A^*|=U|A|U^*$ and
		\[
		\begin{pmatrix}
			|A^*| & A\\
			A^* & |A|
		\end{pmatrix}
		=
		\begin{pmatrix}
			U|A|U^* & U|A|\\
			|A|U^* & |A|
		\end{pmatrix}
		=
		\begin{pmatrix}
			U|A|^{\frac{1}{2}}\\
			|A|^{\frac{1}{2}}
		\end{pmatrix}
		\begin{pmatrix}
			|A|^{\frac{1}{2}}U^* & |A|^{\frac{1}{2}}
		\end{pmatrix}\ge 0.
		\]
	\end{proof}

Now,	we  record the optimal constant in Lee's inequality \eqref{eq:Lee-uin} when $m\ge n$.
	
	\begin{thm}\label{thm:lee-optimal-n-less-m}
 Let $m\ge n$ and
		$A_1,\dots,A_m\in\mathbb M_n$. Then
		\begin{equation}\label{eq:Lee-opt-m>n}
			\Bigl\|\sum_{k=1}^m A_k\Bigr\|\ \le\ \sqrt n\,
			\Bigl\|\sum_{k=1}^m |A_k|\Bigr\|.
		\end{equation}
		Moreover, the constant $\sqrt n$ is sharp (already for the operator norm).
	\end{thm}
	
	\begin{proof}
		Set
		\[
		S:=\sum_{k=1}^m A_k,\qquad X:=\sum_{k=1}^m |A_k|,\qquad Y:=\sum_{k=1}^m |A_k^*|.
		\]
		By Lemma~\ref{lem:pos-block-MA}, for each $k$ we have
		\[
		\begin{pmatrix}
			|A_k^*| & A_k\\
			A_k^* & |A_k|
		\end{pmatrix}\ge 0.
		\]
		Summing over $k$ yields the positive semidefinite block matrix
		\begin{equation*}
			\begin{pmatrix}
				Y & S\\
				S^* & X
			\end{pmatrix}\ge 0.
		\end{equation*}
		By the contraction factorization (Lemma~\ref{lem:contraction}), there exists a contraction $K$
		(with $\|K\|_\infty\le 1$) such that
		\begin{equation}\label{eq:factor-S}
			S=Y^{\frac{1}{2}}KX^{\frac{1}{2}}.
		\end{equation}
		
		Fix $1\le r\le n$ and consider the Ky Fan $r$-norm $\|\cdot\|_{(r)}$.
		Using the variational formula (e.g., see \cite{Bha97})
		\[
		\|S\|_{(r)}
		=\max\Big\{|\tr(U^*SV)|:\ U,V\in\mathbb C^{n\times r},\ U^*U=V^*V=I_r\Big\},
		\]
	by using H\"{o}lder inequality (Lemma~\ref{lem:Holder-inequality}),	for such $U,V$, we have the estimate
		\[
		|\tr(U^*SV)|
		=|\tr(U^*Y^{\frac{1}{2}}KX^{\frac{1}{2}}V)|
		\le \|K\|_\infty\,\|Y^{\frac{1}{2}}U\|_2\,\|X^{\frac{1}{2}}V\|_2
		\le \sqrt{\tr(U^*YU)}\,\sqrt{\tr(V^*XV)}.
		\]
		Taking the maximum gives
		\begin{equation}\label{eq:KF-step1}
			\|S\|_{(r)}\ \le\ \sqrt{\ \|Y\|_{(r)}\ \|X\|_{(r)}\ }.
		\end{equation}
		
		Next, since $Y\ge 0$ we have $\|Y\|_{(r)}\le \tr(Y)$, and
		\[
		\tr(Y)=\sum_{k=1}^m \tr(|A_k^*|)=\sum_{k=1}^m \tr(|A_k|)=\tr(X).
		\]
		On the other hand, for $X\ge 0$ with eigenvalues $\lambda_1(X)\ge\cdots\ge\lambda_n(X)\ge 0$,
		\[
		\|X\|_{(r)}=\sum_{j=1}^r \lambda_j(X)\ \ge\ r\cdot \frac{\sum_{j=1}^n \lambda_j(X)}{n}
		=\frac{r}{n}\tr(X),
		\]
		hence $\tr(X)\le \frac{n}{r}\|X\|_{(r)}$. Combining these facts gives
		\begin{equation}\label{eq:YbyX}
			\|Y\|_{(r)}\le \tr(Y)=\tr(X)\le \frac{n}{r}\|X\|_{(r)}.
		\end{equation}
		Substituting \eqref{eq:YbyX} into \eqref{eq:KF-step1} yields
		\[
		\|S\|_{(r)}\le \sqrt{\frac{n}{r}}\ \|X\|_{(r)}\le \sqrt n\,\|X\|_{(r)}
		\qquad (r=1,\dots,n).
		\]
		By Fan's dominance theorem (see, e.g., \cite[p.~93]{Bha97}), this Ky Fan majorization implies
		\[
		\|S\|\le \sqrt n\,\|X\|
		\]
		for every unitarily invariant norm, proving \eqref{eq:Lee-opt-m>n}.
		
		\smallskip
		\noindent\emph{Sharpness.}
		Let $m>n$ and set $A_k:=E_{1k}\in\mathbb M_n$ for $k=1,\dots,n$, and $A_{n+1}=\cdots=A_m:=0$.
		Then $|A_k|=E_{kk}$ for $k\le n$, so $\|\sum_{k=1}^m |A_k|\|_\infty=1$.
		Moreover $\sum_{k=1}^m A_k=\sum_{k=1}^n E_{1k}$ has largest singular value $\sqrt n$, hence
		$\|\sum_{k=1}^m A_k\|_\infty=\sqrt n$. Therefore the ratio equals $\sqrt n$ and the constant
		cannot be improved.
	\end{proof}

	\begin{proof}[Proof of Theorem~\ref{thm:Lee-optimal}]	Combining \eqref{eq:Lee-uin}, Example~\ref{eg:Lee-uin-sharp} and Theorem~\ref{thm:lee-optimal-n-less-m} gives our desired result.
		
	\smallskip
	\noindent\emph{Sharpness.}
	Let $\ell:=\min\{m,n\}$ and set $A_k:=E_{1k}\in\mathbb M_n$ for $k=1,\dots,\ell$, and
	$A_{\ell+1}=\cdots=A_m:=0$.
	Then $\bigl\|\sum_{k=1}^m |A_k|\bigr\|_\infty=1$ and
	$\bigl\|\sum_{k=1}^m A_k\bigr\|_\infty=\sqrt{\ell}$.
	Therefore the factor $\sqrt{\min\{m,n\}}$ in the inequality cannot be improved.
	\end{proof}

	\section{Proofs of Theorems~\ref{thm:sum-vs-sym}, \ref{thm:sum-vs-qsym} and \ref{thm:qsym-Lee}} \label{sec:sum-vs}
We begin with two lemmas.
	\begin{lem}\label{lem:eqdiag-dom}
		Let $P\in\mathbb M_n$ be positive semidefinite and let $Z\in\mathbb M_n$.
		If $
		\begin{pmatrix}
			P & Z\\
			Z^* & P
		\end{pmatrix}\ge 0$,
		then 
		\[
		\|Z\|\le \|P\|.
		\]
	\end{lem}
	
	\begin{proof}
		Fix $x,y\in\mathbb C^n$. Choose $\theta\in\mathbb R$ so that
		$e^{i\theta}x^*Zy=-|x^*Zy|$. Positivity gives
		\[
		0\le
		\begin{pmatrix}x\\ e^{i\theta}y\end{pmatrix}^*
		\begin{pmatrix}P&Z\\ Z^*&P\end{pmatrix}
		\begin{pmatrix}x\\ e^{i\theta}y\end{pmatrix}
		= x^*Px+y^*Py-2|x^*Zy|.
		\]
		Hence
		\begin{equation}\label{eq:bilinear}
			|x^*Zy|\ \le\ \frac{x^*Px+y^*Py}{2}
			\qquad(x,y\in\mathbb C^n).
		\end{equation}
		
		Let $\|\cdot\|_{(k)}$ be the Ky Fan $k$-norm. Using the variational formula
		\[
		\|Z\|_{(k)}
		=\max\Big\{|\tr(U^*ZV)|:\ U,V\in\mathbb C^{n\times k},\ U^*U=V^*V=I_k\Big\},
		\]
		fix such $U=[u_1,\dots,u_k]$ and $V=[v_1,\dots,v_k]$. Then by \eqref{eq:bilinear},
		\[
		|\tr(U^*ZV)|
		\le \sum_{j=1}^k |u_j^*Zv_j|
		\le \frac12\sum_{j=1}^k (u_j^*Pu_j+v_j^*Pv_j)
		=\frac12\bigl(\tr(U^*PU)+\tr(V^*PV)\bigr).
		\]
		By Ky Fan's maximum principle,
		$\tr(U^*PU)\le \sum_{j=1}^k \lambda_j(P)$ and likewise for $V$.
		Therefore,
		\[
		\|Z\|_{(k)}\le \sum_{j=1}^k \lambda_j(P)=\|P\|_{(k)}
		\qquad (k=1,\dots,n).
		\]

		Finally, Fan's dominance theorem \cite[p.~93]{Bha97} implies that
		$\|Z\|\le\|P\|$ for every unitarily invariant norm.
	\end{proof}

	\begin{proof}[Proof of Theorem~\ref{thm:sum-vs-sym}]
		For each $k$,  by Lemma~\ref{lem:pos-block-MA}, we have
		\[
		\begin{pmatrix}
			2|A_k|_{\mathrm{sym}} & A_k\\
			A_k^* & 2|A_k|_{\mathrm{sym}}
		\end{pmatrix}
		=
		\begin{pmatrix}
			|A_k^*| & A_k\\
			A_k^* & |A_k|
		\end{pmatrix}
		+
		\begin{pmatrix}
		|A_k| & 0\\
			0 & |A_k^*|
		\end{pmatrix}
		\ge 0.
		\]
		Summing over $k$ yields
		\[
		\begin{pmatrix}
			2\sum_{k=1}^m|A_k|_{\mathrm{sym}} & \sum_{k=1}^m A_k\\
			(\sum_{k=1}^m A_k)^* & 2\sum_{k=1}^m|A_k|_{\mathrm{sym}}
		\end{pmatrix}\ge 0.
		\]
		Applying Lemma~\ref{lem:eqdiag-dom} gives
		\[
		\Bigl\|\sum_{k=1}^m A_k\Bigr\|
	\le 2\Bigl\|\sum_{k=1}^m |A_k|_{\mathrm{sym}}\Bigr\|.
		\]
		
		\smallskip
		\noindent\emph{Sharpness.}
		Take $n=2$ and $A_1=\cdots=A_m=\begin{pmatrix}
			0&1\\
			0&0
		\end{pmatrix}$ with the operator norm.
		Then $\|\sum_k A_k\|_\infty=m$ while
		$\|\sum_k |A_k|_{\mathrm{sym}}\|_\infty=\frac{m}{2}$, so the ratio equals $2$.
	\end{proof}
	
	\begin{proof}[Proof of Theorem~\ref{thm:sum-vs-qsym}]
	For each $k$,	by using Lemma~\ref{lem:pos-block-MA}, we obtain
		\[
	\begin{pmatrix}
		\sqrt{2}|A_k|_{\mathrm{qsym}} & A_k\\
		A_k^* & \sqrt{2}|A_k|_{\mathrm{qsym}}
	\end{pmatrix}
	=
	\begin{pmatrix}
		|A_k^*| & A_k\\
		A_k^* & |A_k|
	\end{pmatrix}
	+
	\begin{pmatrix}
		\sqrt{2}|A_k|_{\mathrm{qsym}}-|A_k^*| & 0\\
		0 & \sqrt{2}|A_k|_{\mathrm{qsym}}-|A_k|
	\end{pmatrix}\ge0, 
		\]where $\sqrt2|A_k|_{\mathrm{qsym}}\ge |A_k|$ and $\sqrt2|A_k|_{\mathrm{qsym}}\ge |A_k^*|$
		follow from $(|A_k|^2+|A_k^*|^2)/2 \ge |A_k|^2/2$ (and similarly for $|A_k^*|^2$)
		by operator monotonicity of $t\mapsto t^{\frac{1}{2}}$.
		Summing over $k$ gives
		\[
		\begin{pmatrix}
			\sqrt2\sum_{k=1}^m |A_k|_{\mathrm{qsym}} & \sum_{k=1}^m A_k\\
			(\sum_{k=1}^m A_k)^* & \sqrt2\sum_{k=1}^m |A_k|_{\mathrm{qsym}}
		\end{pmatrix}\ge 0.
		\]
		Applying Lemma~\ref{lem:eqdiag-dom} yields
		\[
		\Bigl\|\sum_{k=1}^m A_k\Bigr\|
	\le 
		\sqrt2\Bigl\|\sum_{k=1}^m |A_k|_{\mathrm{qsym}}\Bigr\|.
		\]
		
		\smallskip
		\noindent\emph{Sharpness.}
		Take $n=2$ and $A_1=\cdots=A_m=\begin{pmatrix}
			0&1\\
			0&0
		\end{pmatrix}$ with the operator norm.
		Since $\left| \begin{pmatrix}
			0&1\\
			0&0
		\end{pmatrix}\right| _{\mathrm{qsym}}=\frac{\sqrt{2}}{2}I_2$, we have
		$\|\sum_k A_k\|_\infty=m$ and $\|\sum_k |A_k|_{\mathrm{qsym}}\|_\infty=\frac{\sqrt{2}}{2}m$,
		so the ratio equals $\sqrt2$.
	\end{proof}
	
\begin{proof}[Proof of Theorem~\ref{thm:qsym-Lee}]
	For $A\in\mathbb M_n$, define the $2n\times 2n$ matrix
	\[
	\Phi(A):=\frac1{\sqrt2}\begin{pmatrix}A&0\\ A^*&0\end{pmatrix}.
	\]
	A direct computation shows
	\begin{equation}\label{eq:Phi-abs}
		|\Phi(A)|=\begin{pmatrix}|A|_{\mathrm{qsym}}&0\\0&0\end{pmatrix}.
	\end{equation}
	Clearly, $	\Phi$ is additive, that is,
\begin{equation}\label{eq:Phi-additive}
	\Phi(\sum_{k=1}^m A_k)=\sum_{k=1}^m \Phi(A_k).
\end{equation}
	
	Applying Theorem~\ref{thm:Lee-optimal} to the matrices $\Phi(A_1),\dots,\Phi(A_m)\in\mathbb M_{2n}$ gives
\begin{equation}\label{eq:Phi-Lee-inequality}
	\Bigl\|\sum_{k=1}^m \Phi(A_k)\Bigr\|
\le \sqrt{\min\{m,2n\}} \Bigl\|\sum_{k=1}^m |\Phi(A_k)|\Bigr\|.
\end{equation}
Thus, 
\begin{align*}
	\left\| \begin{pmatrix}|\sum_{k=1}^m A_k|_{\mathrm{qsym}}&0\\0&0\end{pmatrix}\right\| &=
		\bigl\||\Phi(\sum_{k=1}^m A_k)|\bigr\|&\text{(By \eqref{eq:Phi-abs})}\\
		&=	\bigl\|\Phi(\sum_{k=1}^m A_k)\bigr\|\\
		&=
		\bigl\|\sum_{k=1}^m \Phi(A_k)\bigr\|&\text{(By \eqref{eq:Phi-additive})}\\
		&\le \sqrt{\min\{m,2n\}}  \Bigl\|\sum_{k=1}^m |\Phi(A_k)|\Bigr\|&\text{(By \eqref{eq:Phi-Lee-inequality})}\\
	&=\sqrt{\min\{m,2n\}}  \Bigl\|\sum_{k=1}^m
	\begin{pmatrix}|A_k|_{\mathrm{qsym}}&0\\0&0\end{pmatrix}\Bigr\| &\text{(By \eqref{eq:Phi-abs})}\\
	&=\sqrt{\min\{m,2n\}}  \Bigl\|
	\begin{pmatrix}\sum_{k=1}^m |A_k|_{\mathrm{qsym}}&0\\0&0\end{pmatrix}\Bigr\|.
\end{align*}

	Finally, for any unitarily invariant norm, adjoining a zero direct summand does not change the norm
	(since it only appends zero singular values). Therefore,
	\[
\left\| \begin{pmatrix}|\sum_{k=1}^m A_k|_{\mathrm{qsym}}&0\\0&0\end{pmatrix}\right\|
	=\bigl\||\sum_{k=1}^m A_k|_{\mathrm{qsym}}\bigr\|,
	\qquad
	\Bigl\|
	\begin{pmatrix}\sum_{k=1}^m |A_k|_{\mathrm{qsym}}&0\\0&0\end{pmatrix}\Bigr\|
	=\Bigl\|\sum_{k=1}^m |A_k|_{\mathrm{qsym}}\Bigr\|.
	\]
	Combining these inequalities yields
	\[
	\Bigl\lVert \Bigl|\sum_{k=1}^m A_k\Bigr|_{\mathrm{qsym}} \Bigr\rVert
	\le \sqrt{\min\{m,2n\}} \,
	\Bigl\lVert \sum_{k=1}^m \bigl|A_k\bigr|_{\mathrm{qsym}} \Bigr\rVert.
	\]
\end{proof}
\section{Partial results on Problem~\ref{prob:general-lee-problem}}\label{sec:partial-results-general-lee}
First, we prove a sharp $\|\cdot\|_1\|\cdot\|_\infty$--$\|\cdot\|_2$ inequality.
\begin{lem}\label{lem:1inf-2-sharp}
	Let $X\in\Mn$ be positive semidefinite. Then
	\begin{equation}\label{eq:1inf-2-sharp}
		\|X\|_1\,\|X\|_\infty \ \le\ \frac{1+\sqrt{n}}{2}\,\|X\|_2^2.
	\end{equation}
	Moreover, the constant $\frac{1+\sqrt{n}}{2}$ is optimal.
\end{lem}

\begin{proof}
	Let $\lambda_1\ge\cdots\ge\lambda_n\ge 0$ be the eigenvalues of $X$. Then
	$\|X\|_1=\sum_{i=1}^n \lambda_i$, $\|X\|_\infty=\lambda_1$, and
	$\|X\|_2^2=\sum_{i=1}^n \lambda_i^2$. Hence \eqref{eq:1inf-2-sharp} is equivalent to
	\begin{equation}\label{eq:eig-ineq}
		\lambda_1\Bigl(\sum_{i=1}^n \lambda_i\Bigr)
		\ \le\ \frac{1+\sqrt{n}}{2}\sum_{i=1}^n \lambda_i^2.
	\end{equation}
	
	Fix $\lambda_1=a\ge 0$ and set $s=\sum_{i=2}^n \lambda_i$. By Cauchy--Schwarz inequality,
	$\sum_{i=2}^n \lambda_i^2 \ge s^2/(n-1)$, so for fixed $(a,s)$ the ratio
	\[
	\frac{a(a+s)}{a^2+\sum_{i=2}^n \lambda_i^2}
	\]
	is maximized when $\lambda_2=\cdots=\lambda_n=b:=s/(n-1)$. Thus it suffices to maximize,
	over $0\le b\le a$,
	\[
	R(a,b):=\frac{a(a+(n-1)b)}{a^2+(n-1)b^2}.
	\]
	By homogeneity we may set $a=1$ and write $r=b/a\in[0,1]$, obtaining
	\[
	R(r)=\frac{1+(n-1)r}{1+(n-1)r^2}.
	\]
	A direct derivative computation gives
	\[
	R'(r)=\frac{(n-1)\bigl(1-2r-(n-1)r^2\bigr)}{\bigl(1+(n-1)r^2\bigr)^2},
	\]
	so the unique critical point in $(0,1)$ is the positive root of
	$(n-1)r^2+2r-1=0$, namely $r=\frac{1}{\sqrt{n}+1}$. Substituting this value yields
	\[
	\max_{r\in[0,1]} R(r)
	= \frac{1+\sqrt{n}}{2}.
	\]
	This proves \eqref{eq:eig-ineq}, hence \eqref{eq:1inf-2-sharp}. Optimality follows
	because equality is attained when $\lambda_1=a$ and
	$\lambda_2=\cdots=\lambda_n=\frac{a}{\sqrt{n}+1}$.
\end{proof}
Next, we prove a sharp inequality on the Frobenius norm for  Problem~\ref{prob:general-lee-problem}.
\begin{thm}\label{thm:c2-minmn}
	Let $m,n\ge 1$ and $A_1,\dots,A_m\in\Mn$. Set $d:=\min\{m,n\}$, then
	\[
	\Bigl\|\sum_{k=1}^m A_k\Bigr\|_2
	\ \le\
	\sqrt{\frac{1+\sqrt{d}}{2}}\;
	\Bigl\|\sum_{k=1}^m |A_k|\Bigr\|_2.
	\]
Moreover, the constant $\sqrt{\frac{1+\sqrt{d}}{2}}$ is sharp.
\end{thm}

\begin{proof}
	Write $S:=\sum_{k=1}^m A_k$ and $X:=\sum_{k=1}^m |A_k|$.
	
	\smallskip
	\noindent\emph{Step 1: a new upper bound.}
	Take polar decompositions $A_k=U_k|A_k|$. Define the block matrices
\begin{align*}
		Z:&=\bigl[\,|A_1|^{\frac{1}{2}}\ \ |A_2|^{\frac{1}{2}}\ \ \cdots\ \ |A_m|^{\frac{1}{2}}\,\bigr]\in\mathbb{M}_{n,mn},\\
	W:&=\bigl[\,U_1|A_1|^{\frac{1}{2}}\ \ \cdots\ \ U_m|A_m|^{\frac{1}{2}}\,\bigr]\in\mathbb{M}_{n,mn}.
\end{align*}
Then $S=WZ^*$ and $ZZ^*=\sum_{k=1}^m |A_k|=X$. Using $\|AB\|_2\le \|A\|_2\|B\|_\infty$ for any $A,B\in \Mn$,
	we get
	\[
	\|S\|_2^2=\|WZ^*\|_2^2\ \le\ \|W\|_2^2\,\|Z\|_\infty^2.
	\]
	Notice that $\|W\|_2^2=\tr(WW^*)=\sum_{k=1}^m \tr(|A_k|)=\tr(X)=\|X\|_1$,
	and $\|Z\|_\infty^2=\|ZZ^*\|_\infty=\|X\|_\infty$. Hence
	\[
	\|S\|_2^2\ \le\ \|X\|_1\,\|X\|_\infty.
	\]
	Applying Lemma~\ref{lem:1inf-2-sharp} gives
	\[
	\|S\|_2^2 \ \le\ \frac{1+\sqrt{n}}{2}\,\|X\|_2^2,
	\qquad\text{i.e.}\qquad
	\|S\|_2 \ \le\ \sqrt{\frac{1+\sqrt{n}}{2}}\;\|X\|_2.
	\]
	Thus $c_2^{\mathrm{abs}}(m,n)\le \sqrt{\frac{1+\sqrt{n}}{2}}$ for all $m$.
	
	\smallskip
	\noindent\emph{Step 2: Tang-Zhang's upper bound.}
	Tang and Zhang \cite{TZ26} proved that for all $n$,
	\[
	c_2^{\mathrm{abs}}(m,n)\le \sqrt{\frac{1+\sqrt{m}}{2}}.
	\]
In particular, for $n\ge m$,
	$c_2^{\mathrm{abs}}(m,n)= \sqrt{\frac{1+\sqrt{m}}{2}}$.
	
	\smallskip
	Combining Steps 1 and 2 yields the uniform upper bound
	\[
	c_2^{\mathrm{abs}}(m,n)\ \le\ \sqrt{\frac{1+\sqrt{\min\{m,n\}}}{2}}.
	\]
	
	\smallskip
	\noindent\emph{Step 3: Sharpness (lower bound).}
	Let $d=\min\{m,n\}$ and set $t:=\frac{1}{\sqrt{d}+1}$. Consider the $d\times d$ matrix
	$G:=(1-t)I_d+tJ_d$, where $J_d$ is the all-ones matrix. Since
	$\sigma(G)=\{\,1-t\ \text{(mult.\ $d-1$)},\ 1+(d-1)t\ \text{(mult.\ $1$)}\,\}\subset(0,\infty)$,
	we have $G\ge 0$. Hence there exists a matrix $V\in\mathbb{M}_{d}$ with $V^*V=G$.
	Let $v_1,\dots,v_d\in\mathbb{C}^d$ be the columns of $V$. Then $\|v_k\|=1$ and
	$\langle v_i,v_j\rangle = t$ for $i\neq j$.
	
	Embed $\mathbb{C}^d$ into $\mathbb{C}^n$ (if $d<n$) by padding zeros, and fix a unit vector
	$u\in\mathbb{C}^n$ (e.g.\ $u=e_1$). Define, for $1\le k\le d$,
	\[
	A_k:=u v_k^*\in\Mn,
	\qquad\text{ and set }	\qquad
	A_{d+1}=\cdots=A_m:=0.
	\]
	Then $|A_k|=(A_k^*A_k)^{\frac{1}{2}}=(v_kv_k^*)^{\frac{1}{2}}=v_kv_k^*$, so
	$X=\sum_{k=1}^m |A_k|=\sum_{k=1}^d v_kv_k^*$. Moreover,
	$S=\sum_{k=1}^m A_k=u\bigl(\sum_{k=1}^d v_k\bigr)^*$, so
	\[
	\|S\|_2^2=\Bigl\|\sum_{k=1}^d v_k\Bigr\|^2
	= \sum_{k=1}^d \|v_k\|^2 + \sum_{i\neq j}\langle v_i,v_j\rangle
	= d + d(d-1)t = d\bigl(1+(d-1)t\bigr).
	\]
	Also,
	\[
	\|X\|_2^2=\tr(X^2)
	= \sum_{k=1}^d \tr\bigl((v_kv_k^*)^2\bigr)
	+ \sum_{i\neq j}\tr(v_iv_i^*v_jv_j^*)
	= d + \sum_{i\neq j}|\langle v_i,v_j\rangle|^2
	= d + d(d-1)t^2 = d\bigl(1+(d-1)t^2\bigr).
	\]
	Therefore
	\[
	\frac{\|S\|_2^2}{\|X\|_2^2}
	= \frac{1+(d-1)t}{1+(d-1)t^2}.
	\]
	With $t=\frac{1}{\sqrt{d}+1}$ one checks that
	$1+(d-1)t=\sqrt{d}$ and $1+(d-1)t^2=\frac{2\sqrt{d}}{\sqrt{d}+1}$, hence
	\[
	\frac{\|S\|_2^2}{\|X\|_2^2}=\frac{1+\sqrt{d}}{2},
	\qquad\text{so}\qquad
	\frac{\|S\|_2}{\|X\|_2}=\sqrt{\frac{1+\sqrt{d}}{2}}.
	\]
	This shows $c_2^{\mathrm{abs}}(m,n)\ge \sqrt{\frac{1+\sqrt{d}}{2}}$, matching the upper bound.
\end{proof}
Next, we give a dimension-dependent upper bound for general $p$ in Problem~\ref{prob:general-lee-problem}.
\begin{prop}\label{prop:cp-min}
	Let $1\le p\le \infty$ and $A_1,\dots,A_m\in\Mn$. Put $d:=\min\{m,n\}$. Then
	\[
	\Bigl\|\sum_{k=1}^m A_k\Bigr\|_p
	\ \le\
	d^{\frac12-\frac{1}{2p}}\;
	\Bigl\|\sum_{k=1}^m |A_k|\Bigr\|_p.
	\]
\end{prop}

\begin{proof}
	Fix $1\le p<\infty$ and set
	\[
	S=\sum_{k=1}^m A_k,\qquad
	X=\sum_{k=1}^m |A_k|,\qquad
	Y=\sum_{k=1}^m |A_k^*|.
	\]
	By Lemma~\ref{lem:pos-block-MA}, for each $k$,
	\[
	\begin{pmatrix}
		|A_k^*| & A_k\\
		A_k^* & |A_k|
	\end{pmatrix}\ge 0.
	\]
	Summing over $k$ gives
	\[
	\begin{pmatrix}
		Y & S\\
		S^* & X
	\end{pmatrix}\ge 0.
	\]
	By Lemma~\ref{lem:contraction}, there exists a contraction $K$ such that
	\[
	S=Y^{\frac{1}{2}}KX^{\frac{1}{2}}.
	\]
	Using H\"{o}lder's inequality for Schatten norms (Lemma~\ref{lem:Holder-inequality}),
	\[
	\|S\|_p
	=\|Y^{\frac{1}{2}}KX^{\frac{1}{2}}\|_p
	\le \|Y^{\frac{1}{2}}\|_{2p}\,\|K\|_\infty\,\|X^{\frac{1}{2}}\|_{2p}
	\le \|Y^{\frac{1}{2}}\|_{2p}\,\|X^{\frac{1}{2}}\|_{2p}.
	\]
	Since $X,Y\ge 0$, we have $\|Y^{\frac{1}{2}}\|_{2p}^2=\|Y\|_p$ and
	$\|X^{\frac{1}{2}}\|_{2p}^2=\|X\|_p$, hence
	\[
	\|S\|_p^2\le \|X\|_p\,\|Y\|_p.
	\]
	Next, $\|Y\|_p\le \|Y\|_1=\tr(Y)=\tr(X)=\|X\|_1$ and
	$\|X\|_1\le n^{1-\frac1p}\|X\|_p$. Therefore
	\[
	\|S\|_p^2\le n^{1-\frac1p}\|X\|_p^2,
	\qquad\text{i.e.}\qquad
	\|S\|_p\le n^{\frac12-\frac{1}{2p}}\|X\|_p.
	\]
	Combining this with Tang-Zhang's bound $c_p^{\mathrm{abs}}(m,n)\le m^{\frac12-\frac{1}{2p}}$
	from \cite{TZ26} yields
	\[
	c_p^{\mathrm{abs}}(m,n)\le (\min\{m,n\})^{\frac12-\frac{1}{2p}}.
	\]
	The case $p=\infty$ follows by letting $p\to\infty$, or
	directly from the operator-norm case.
\end{proof}
	\section{Proofs of Theorems~\ref{thm:schatten-sum-vs-sym-p} and \ref{thm:schatten-sum-vs-qsym-p}}\label{sec:proofs-schatten-p}
The classic  McCarthy's inequality \cite{McC67} states that
\begin{lem}\label{lem:McCarthy}
	Let $A,B\in \Mn$ be positive semidefinite. Then for all $p\ge 1$,
	\[
	\tr (A^p + B^p) \le \tr (A+B)^p.
	\]
\end{lem}
	\begin{proof}[Proof of Theorem~\ref{thm:schatten-sum-vs-sym-p}]
		Let $1\le p\le\infty$ and $A_1,\dots,A_m\in\mathbb M_n$. Set
		\[
		S:=\sum_{k=1}^m A_k,\qquad
		X:=\sum_{k=1}^m |A_k|,\qquad
		Y:=\sum_{k=1}^m |A_k^*|.
		\]
	Hence,
	$$\sum_{k=1}^m |A_k|_{\mathrm{sym}}=\frac{X+Y}{2}.$$

		For each $k$, define the Hermitian dilation
		\[
		H_k:=\begin{pmatrix}0&A_k\\ A_k^*&0\end{pmatrix}\in\mathbb M_{2n},
		\qquad
		H:=\sum_{k=1}^m H_k=\begin{pmatrix}0&S\\ S^*&0\end{pmatrix}.
		\]
		A direct computation yields
		\[
		|H_k|=\begin{pmatrix}|A_k^*|&0\\0&|A_k|\end{pmatrix},
		\qquad
		D:=\sum_{k=1}^m |H_k|
		=\begin{pmatrix}Y&0\\0&X\end{pmatrix}.
		\]
		
	For every Hermitian matrix $H_k$, by Lemma~\ref{lem:pos-block-MA}, we know
		\begin{equation*}
			\begin{pmatrix}|H_k|&H_k\\ H_k&|H_k|\end{pmatrix}\ge 0.
		\end{equation*}
	Summing over $k$ gives
		\[
		\begin{pmatrix}D&H\\ H&D\end{pmatrix}
		=\sum_{k=1}^m \begin{pmatrix}|H_k|&H_k\\ H_k&|H_k|\end{pmatrix}\ge 0.
		\]
	Thus by Lemma~\ref{lem:contraction}, $H=D^\frac{1}{2}KD^\frac{1}{2}$ for some contraction $K$.
		By using H\"{o}lder inequality (Lemma~\ref{lem:Holder-inequality}),
\begin{equation}\label{eq:H-D-p-Holder}
	\|H\|_p
=\|D^\frac{1}{2}KD^\frac{1}{2}\|_p
\le \|D^\frac{1}{2}\|_{2p}\,\|K\|_\infty\,\|D^\frac{1}{2}\|_{2p}
\le \|D^\frac{1}{2}\|_{2p}^2=\|D\|_{p}.
\end{equation}	
A direct computation shows that
$
|H|=\begin{pmatrix}|S^*|&0\\0&|S|\end{pmatrix},
$
hence 
\begin{equation}\label{eq:Schatten-p-H}
	\|H\|_p=2^{\frac{1}{p}}\|S\|_p=2^{\frac{1}{p}}\|\sum_{k=1}^m A_k\|_p.
\end{equation}

By McCarthy's inequality (Lemma~\ref{lem:McCarthy}),  we have
	\begin{align}\label{eq:Schatten-p-D}
		\|D\|_p&=(\tr X^p+\tr Y^p)^\frac{1}{p}\nonumber\\
		&\le \left( \tr (X+Y)^p\right)^\frac{1}{p}\nonumber\\
		&=\|X+Y\|_p\nonumber\\
		&=\left\| \sum_{k=1}^m |A_k|+\sum_{k=1}^m |A_k^*|\right\|_p\nonumber\\
		&=2\left\| \sum_{k=1}^m |A_k|_{\mathrm{sym}}\right\|_p.
	\end{align}
Combining \eqref{eq:H-D-p-Holder}, \eqref{eq:Schatten-p-H} and \eqref{eq:Schatten-p-D} gives
		\[
		\Bigl\|\sum_{k=1}^m A_k\Bigr\|_p
		\le 2^{1-\frac{1}{p}}\Bigl\|\sum_{k=1}^m |A_k|_{\mathrm{sym}}\Bigr\|_p.
		\]
		This proves the claim.
		
\medskip
\noindent\emph{Sharpness.}
Let $n=2$ and let $E_{12}\in\mathbb{M}_2$ be the matrix unit.
For each $m\ge1$, set $A_1=\cdots=A_m:=E_{12}$.
Then $\bigl\|\sum_{k=1}^m A_k\bigr\|_p=\|mE_{12}\|_p=m$ for every $1\le p\le\infty$.

Moreover, $|E_{12}|=\diag(0,1)$ and $|E_{12}^*|=\diag(1,0)$, hence
$|E_{12}|_{\mathrm{sym}}=(|E_{12}|+|E_{12}^*|)/2=I_2/2$.
Therefore
\[
\Bigl\|\sum_{k=1}^m |A_k|_{\mathrm{sym}}\Bigr\|_p
=\Bigl\|\frac{m}{2}I_2\Bigr\|_p
=
\begin{cases}
	m\,2^{\frac1p-1}, & 1\le p<\infty,\\[2pt]
	\frac{m}{2}, & p=\infty.
\end{cases}
\]
Consequently,
\[
\frac{\bigl\|\sum_{k=1}^m A_k\bigr\|_p}{\bigl\|\sum_{k=1}^m |A_k|_{\mathrm{sym}}\bigr\|_p}
=2^{1-\frac1p},
\]
so equality holds in Theorem~\ref{thm:schatten-sum-vs-sym-p} and the constant $2^{1-\frac1p}$ is sharp.
	\end{proof}
	\begin{lem}\label{lem:trace-cs}
		For all $A,B\in\Mn$,
		\begin{equation}\label{eq:trace-cs}
			|\tr(A^*B)|^2\ \le\ \tr(\absv{A}\,\absv{B})\,\tr(\absv{A^*}\,\absv{B^*}).
		\end{equation}
		Consequently,
		\begin{equation}\label{eq:trace-amgm}
			|\tr(A^*B)|\ \le\ \frac{\tr(\absv{A}\,\absv{B})+\tr(\absv{A^*}\,\absv{B^*})}{2}.
		\end{equation}
	\end{lem}
	
	\begin{proof}
		Let $A=U\absv{A}$ and $B=V\absv{B}$ be polar decompositions.
		Then 
		\begin{align*}
			|\tr(A^*B)|&=|\tr(\absv{A}U^*V\absv{B})| \\
			&=|\tr\Big((\absv{A}^\frac{1}{2}U^*V\absv{B}^\frac{1}{2})(\absv{B}^\frac{1}{2}\absv{A}^\frac{1}{2})\Big)|\\
			&\le\left\| \absv{A}^\frac{1}{2}U^*V\absv{B}^\frac{1}{2}\right\| _2\left\| \absv{B}^\frac{1}{2}\absv{A}^\frac{1}{2}\right\| _2\text{ (By Cauchy-Schwarz inequality)}\\
			&=\left( \tr(\absv{A^*}\,\absv{B^*})\right)^\frac{1}{2} \left( \tr(\absv{A}\,\absv{B})\right)^\frac{1}{2} 
		\end{align*}
which is equivalent to \eqref{eq:trace-cs}.

		The bound \eqref{eq:trace-amgm} follows from \eqref{eq:trace-cs} by the scalar AM--GM inequality.
	\end{proof}
		Lieb's theorem \cite[Theorem~1]{Lie73} asserts that the map $(X,Y)\mapsto \tr(X^\frac{1}{2}Y^\frac{1}{2})$ is jointly concave on $\Mn^+\times\Mn^+$, where $\Mn^+$ is the positive cone on $\Mn$. That is,
	\begin{lem}\label{lem:lieb}
		Let $X_1,X_2,Y_1,Y_2\in\Mn$ be positive semidefinite. Then
		\begin{equation*}
			\tr\!\left(\left(\frac{X_1+X_2}{2}\right)^\frac{1}{2}\left(\frac{Y_1+Y_2}{2}\right)^\frac{1}{2}\right)
			\ \ge\ \frac{\tr(X_1^\frac{1}{2}Y_1^\frac{1}{2})+\tr(X_2^\frac{1}{2}Y_2^\frac{1}{2})}{2}.
		\end{equation*}
	\end{lem}

Next, we present a sharp inequality for the Frobenius norm. Although it is subsumed by Theorem~\ref{thm:schatten-sum-vs-qsym-p}, the argument in this special case is particularly short and elegant.
	
	\begin{thm}\label{thm:frob-lee}
		Let $A_1,\dots,A_m\in\Mn$. Then
		\begin{equation}		\label{eq:frob-qsym}
			\hsnorm{\sum_{k=1}^m A_k}
			\le\ \hsnorm{\sum_{k=1}^m \qsma{A_k}},
		\end{equation}
		Moreover, the constant $1$ is sharp.
	\end{thm}
	
	\begin{proof}
 Clearly,
		\begin{equation}\label{eq:Pk-frob}
			\|\qsma{A_k}\|_2^2=\|A_k\|_2^2.
		\end{equation}
		For any $i\neq j$, by using Lemma~\ref{lem:trace-cs} and Lemma~\ref{lem:lieb} with
		$X_1=\absv{A_i}^2$, $X_2=\absv{A_i^*}^2$, $Y_1=\absv{A_j}^2$ and $Y_2=\absv{A_j^*}^2$, we have
	\begin{align*}
			|\tr(A_i^*A_j)|
		&\le \frac{\tr(\absv{A_i}\,\absv{A_j})+\tr(\absv{A_i^*}\,|A_j^*|)}{2}\\
&\le \tr\!\left(\left(\frac{|A_i|^2+\absv{A_i^*}^2}{2}\right)^\frac{1}{2}\left(\frac{\absv{A_j}^2+|A_j^*|^2}{2}\right)^\frac{1}{2}\right)\\
&=\tr\left(\qsma{A_i}\qsma{A_j}\right). 
	\end{align*}
	Hence,
	\begin{align*}
		\|\sum_{k=1}^m A_k\|_2^2
		&=\sum_{k=1}^m \|A_k\|_2^2+2\sum_{i<j}\Re\,\tr(A_i^*A_j) \\
		&\le\sum_{k=1}^m \|A_k\|_2^2+2\sum_{i<j}\left| \tr(A_i^*A_j)\right| \\
		&\le
		\sum_{k=1}^m \|\qsma{A_k}\|_2^2+2\sum_{i<j}\tr\left(\qsma{A_i}\qsma{A_j}\right)\\
		&=\Big\|\sum_{k=1}^m \qsma{A_k}\Big\|_2^2,
	\end{align*}
		where we used \eqref{eq:Pk-frob}. Taking square roots yields \eqref{eq:frob-qsym}.
					
		\end{proof}
		
\begin{rem}\label{rem:trace-qsym}
	The above argument yields the following elegant trace inequality:
\begin{equation}\label{eq:trace-key-qsym}
		|\tr(A^*B)|\ \le\ \tr\!\big(|A|_{\mathrm{qsym}}\,|B|_{\mathrm{qsym}}\big),
	\qquad \text{for all } A,B\in\mathbb M_n.
\end{equation}
\end{rem}
Aujla and Silva~\cite{AS03} established the following inequality for unitarily invariant norms associated with convex functions.
\begin{lem}[{\cite[Corollary~2.6]{AS03}}]\label{lem:Aujla-Silva}
	Let $A,B\in \Mn$ be positive semidefinite and $f$ is a nonnegative convex function defined on $[0,\infty)$. Then for all $\alpha\in [0,1]$,
	\[
	\left\|f(\alpha A+(1-\alpha)B) \right\| \le 	\left\|\alpha f( A)+(1-\alpha)f(B) \right\|.
	\]
	In particular, for the trace norm, we have
	\[
\tr	f(\alpha A+(1-\alpha)B)\le  \tr \left( \alpha f( A)+(1-\alpha)f(B) \right).
	\]
\end{lem}
	\begin{proof}[Proof of Theorem~\ref{thm:schatten-sum-vs-qsym-p}]  Set
	\[
		S:=\sum_{k=1}^m A_k,\qquad
	Q:=\sum_{k=1}^m |A_k|_{\mathrm{qsym}}.
	\]
		Let $q$ be the conjugate index of $p$, i.e.\ $\frac1p+\frac1{q}=1$.
		We use the duality formula for Schatten $p$-norms:
		\begin{equation}\label{eq:dual-schatten}
			\|S\|_p=\sup\Big\{\,|\tr(S^*X)|:\ X\in\mathbb M_n,\ \|X\|_{q}\le 1\,\Big\}.
		\end{equation}
		
		\medskip
		\noindent\emph{Step 1: Reduction to a single-matrix constant.}
		Fix $X$ with $\|X\|_{q}\le 1$. Then we have
\begin{align*}
	|\tr(S^*X)|&= |\sum_{k=1}^m \tr(A_k^*X)|\\
	&\le \sum_{k=1}^m |\tr(A_k^*X)|\\
	&\le \sum_{k=1}^m \tr\big(|A_k|_{\mathrm{qsym}}\,|X|_{\mathrm{qsym}}\big) \text{ (By \eqref{eq:trace-key-qsym})}\\
	&= \tr\big(Q\,|X|_{\mathrm{qsym}}\big)\\
	&\le \|Q\|_p\,\||X|_{\mathrm{qsym}}\|_{q} \text{ (By H\"{o}lder inequality)}.
\end{align*}
		Taking the supremum over $\|X\|_{q}\le 1$ and using \eqref{eq:dual-schatten} gives
		\begin{equation}\label{eq:cp-reduction}
			\|S\|_p\ \le\ d_{q}\,\|Q\|_p,
			\qquad
			d_q:=\sup_{X\neq 0}\frac{\||X|_{\mathrm{qsym}}\|_q}{\|X\|_q}.
		\end{equation}
		It remains to compute $d_q$.
		
		\medskip
		\noindent\emph{Step 2: Computing $d_q$ (phase transition at $q=2$).}
		Write
		\[
		|X|_{\mathrm{qsym}}^2=\frac{|X|^2+|X^*|^2}{2}=\frac{B+C}{2},
		\qquad B:=|X|^2,\quad C:=|X^*|^2.
		\]
		Then
		\begin{equation}\label{eq:dq-expand}
			\||X|_{\mathrm{qsym}}\|_q^q
			=\tr\!\left(\left(\frac{B+C}{2}\right)^{\frac{q}{2}}\right).
		\end{equation}
		
		\smallskip
		\noindent\emph{Case 1: $q\ge 2$.}
	Since the function $t\mapsto t^\frac{q}{2}$ is  convex on $[0,\infty)$,  Lemma~\ref{lem:Aujla-Silva} yields
	\begin{equation}\label{eq:B-C-trace-convex}
			\tr \left(\frac{B+C}{2}\right)^\frac{q}{2} \le 	\tr \frac{B^\frac{q}{2}+C^\frac{q}{2}}{2}=\|X\|_q^q.
	\end{equation}
	Combining \eqref{eq:dq-expand} and \eqref{eq:B-C-trace-convex} gives
		\[
		\||X|_{\mathrm{qsym}}\|_q^q
		\le 
		\|X\|_q^q,
		\]
		which implies  $d_q\le 1$ for $q\ge 2$. 
		
		Take $X$ to be  normal, then $|X|_{\mathrm{qsym}}=|X|$, this means in \eqref{eq:cp-reduction} $$d_q=1, \qquad 2\le q\le \infty.$$  Therefore
		\[
		\Bigl\lVert \sum_{k=1}^m A_k \Bigr\rVert_p
		\ \le\ \,
		\Bigl\lVert \sum_{k=1}^m \bigl|A_k\bigr|_{\mathrm{qsym}} \Bigr\rVert_p, \quad\quad\quad\,\, 1\le p\le 2.
		\]
		
		\medskip
		\noindent\emph{Sharpness.} For $1\le p\le 2$, take each $A_k$ is positive semidefinite, then $|A_k|_{\mathrm{qsym}}=A_k$,
		giving equality.

		\medskip
		\noindent\emph{Case 2: $1\le q\le 2$.}
		Since $t\mapsto t^\frac{q}{2}$ is concave on $[0,\infty)$.
	by using Lemma~\ref{lem:Bourin-Uchiyama}, we have
		\begin{equation}\label{eq:rotfeld}
			\tr\big((B+C)^\frac{q}{2}\big)\ \le\ \tr(B^\frac{q}{2})+\tr(C^\frac{q}{2})=2\|X\|_q^q.
		\end{equation}
		Combining \eqref{eq:dq-expand} and \eqref{eq:rotfeld} gives
		\[
		\||X|_{\mathrm{qsym}}\|_q\le 2^{\frac1q-\frac12}\|X\|_q,
		\qquad 1\le q\le 2.
		\]
		Therefore $d_q\le 2^{\frac1q-\frac12}$ for $1\le q\le 2$. 
  Take $X=E_{12}\in\mathbb M_2$. Then
		$|E_{12}|_{\mathrm{qsym}}=\frac{\sqrt2}{2}I_2$ and $\|E_{12}\|_q=1$, so
		\[
		\frac{\||E_{12}|_{\mathrm{qsym}}\|_q}{\|E_{12}\|_q}
		=\Big(2\Big(\frac{\sqrt2}{2}\Big)^q\Big)^{1/q}
		=2^{\frac1q-\frac12}.
		\]
		Hence in \eqref{eq:cp-reduction},
		\begin{equation*}
			d_q=2^{\frac1q-\frac12},\qquad 1\le q\le 2.
		\end{equation*}
	Thus,
		\[
		\Bigl\lVert \sum_{k=1}^m A_k \Bigr\rVert_p
		\ \le\ 2^{\frac{1}{2}-\frac{1}{p}}\,
		\Bigl\lVert \sum_{k=1}^m \bigl|A_k\bigr|_{\mathrm{qsym}} \Bigr\rVert_p, \quad 2\le p\le \infty.
		\]
		
		\medskip\noindent\emph{Sharpness}.
		For $2\le p\le\infty$, take $n=2$ and $A_1=\cdots=A_m=E_{12}$. Then $\Bigl\lVert \sum_k A_k \Bigr\rVert_p=m$ while
		$|E_{12}|_{\mathrm{qsym}}=\frac{\sqrt2}{2}I_2$, so
		\[
\Bigl\lVert \sum_{k=1}^m \bigl|A_k\bigr|_{\mathrm{qsym}} \Bigr\rVert_p
		=\left\|\frac{m\sqrt2}{2}I_2\right\|_p
		=m\left(2\left(\frac{\sqrt2}{2}\right)^p\right)^{1/p}
		=m\,2^{\frac1p-\frac12},
		\]
	 showing that the constant $2^{\frac{1}{2}-\frac{1}{p}}$ cannot be improved.
	\end{proof}
	\section{Partial results on Problems~\ref{prob:schatten-sym-vs-sym-p} and \ref{prob:schatten-qsym-vs-qsym-p}}\label{sec:partial-schatten-p}
First, we state some	endpoint values and general bounds for $c_p^{\mathrm{sym}}(m,n)$ in Problem~\ref{prob:schatten-sym-vs-sym-p}.
	\begin{thm}\label{thm:cp-endpoints-bounds}
		Let $A_1,\dots,A_m\in\mathbb M_n$, and let $1\le p\le\infty$. Let $c_p^{\mathrm{sym}}(m,n)$ to be the smallest
		constant such that
		\[
		\Bigl\lVert \Bigl|\sum_{k=1}^m A_k\Bigr|_{\mathrm{sym}} \Bigr\rVert_p
		\ \le\ c_p^{\mathrm{sym}}(m,n)\,
		\Bigl\lVert \sum_{k=1}^m |A_k|_{\mathrm{sym}} \Bigr\rVert_p.
		\]
		Then the following statements hold:
		\begin{enumerate}
	\item		$c_1^{\mathrm{sym}}(m,n)=1$;
	\item  $	c_\infty^{\mathrm{sym}}(m,n)=\sqrt2$ for $n\ge 3$;
	\item $c_p^{\mathrm{sym}}(m,n)\le \min\{2^{1-\frac{1}{p}}, \sqrt{2}\}$;
	\item $	c_p^{\mathrm{sym}}(m,n)\ge \left(\frac{2^{\frac{p}{2}}+2^{1-\frac{p}{2}}}{3}\right)^{\frac{1}{p}}$;
in particular, the right-hand side increases from $1$ at $p=2$ to $\sqrt2$ as $p\to\infty$.
		\end{enumerate}
	\end{thm}
	\begin{proof} \emph{Step 1: $c_1^{\mathrm{sym}}(m,n)=1$.}
For each $k$, we have
\[
\||A_k|_{\mathrm{sym}}\|_1=\|A_k\|_1.
\]
		Hence
		\[
		\Bigl\|\Bigl|\sum_{k=1}^m A_k\Bigr|_{\mathrm{sym}}\Bigr\|_1
		=\Bigl\|\sum_{k=1}^m A_k\Bigr\|_1
		\le \sum_{k=1}^m \|A_k\|_1
		=\sum_{k=1}^m \||A_k|_{\mathrm{sym}}\|_1
		=\Bigl\|\sum_{k=1}^m |A_k|_{\mathrm{sym}}\Bigr\|_1.
		\]
		Thus $c_1^{\mathrm{sym}}(m,n)\le1$. Sharpness follows by taking $A_1\neq0$ and $A_2=\cdots=A_m=0$.
		
		\medskip
		\noindent\emph{Step 2: $c_\infty^{\mathrm{sym}}(m,n)=\sqrt{2}$ for $n\ge 3$.}
		The bound $c_\infty^{\mathrm{sym}}(m,n)\le\sqrt2$ is a special case of Theorem~\ref{thm:Bourin-Lee}.
Sharpness follows by Theorem~\ref{thm:sharp}.
		
	\medskip
	\noindent\emph{Step 3. A universal upper bound.}  Notice that 
	\begin{align*}
		\Bigl\|\Bigl|\sum_{k=1}^m A_k\Bigr|_{\mathrm{sym}}\Bigr\|_p&= \frac{1}{2}\Bigl\|\Bigl|\sum_{k=1}^m A_k\Bigr|+\Bigl|\big(\sum_{k=1}^m A_k\big)^*\Bigr|\Bigr\|_p\\
		&\le 	\Bigl\|\sum_{k=1}^m A_k\Bigr\|_p \text{ (By triangle inequality for $\|\cdot\|_p$)}\\
		&\le 2^{1-\frac{1}{p}}\Bigl\|\sum_{k=1}^m |A_k|_{\mathrm{sym}}\Bigr\|_p\text{ (By Theorem~\ref{thm:schatten-sum-vs-sym-p})}.
	\end{align*}
		This gives $c_p^{\mathrm{sym}}(m,n)\le 2^{1-\frac{1}{p}}$. Combining it with Theorem~\ref{thm:Bourin-Lee} gives the claim.
		
		\medskip
		\noindent\emph{Step 4: a lower bound on $c_p^{\mathrm{sym}}(m,n)$ for $n\ge 3$ via an explicit $3\times3$ construction.} 
	Let $n=3$ and consider the unit vectors
		\[
		u=\frac1{\sqrt3}\begin{pmatrix}1\\1\\1\end{pmatrix},\quad
		v=\frac1{\sqrt3}\begin{pmatrix}1\\-1\\-1\end{pmatrix},\quad
		x=\frac1{\sqrt3}\begin{pmatrix}-1\\1\\-1\end{pmatrix},\quad
		y=\frac1{\sqrt3}\begin{pmatrix}1\\1\\-1\end{pmatrix},
		\]
		and define rank-one matrices $A:=uv^\top$ and $B:=xy^\top$ (real matrices).
		A direct computation shows
		\[
		|A|_{\mathrm{sym}}+|B|_{\mathrm{sym}}=\frac23\,I_3,
		\]
		while $|A+B|_{\mathrm{sym}}$ has eigenvalues
		\[
		\frac23\cdot\sqrt2,\quad \frac23\cdot\frac1{\sqrt2},\quad \frac23\cdot\frac1{\sqrt2}.
		\]
		Therefore, for every $p\ge1$,
	\[
	\frac{\left\|\,|A+B|_{\mathrm{sym}}\,\right\|_p}
	{\left\|\,|A|_{\mathrm{sym}}+|B|_{\mathrm{sym}}\,\right\|_p}
	=
	\frac{\frac{2}{3}\,\bigl(2^{\frac{p}{2}}+2^{1-\frac{p}{2}}\bigr)^{\frac{1}{p}}}
	{\frac{2}{3}\,3^{\frac{1}{p}}}
	=
	\left(\frac{2^{\frac{p}{2}}+2^{1-\frac{p}{2}}}{3}\right)^{\frac{1}{p}}.
	\]
	\end{proof}
For completeness, we record a concrete $2\times 2$ counterexample showing that the Frobenius-norm inequality
\[
\|\,|A+B|_{\mathrm{sym}}\|_2 \le \|\,|A|_{\mathrm{sym}}+|B|_{\mathrm{sym}}\|_2
\]
fails in general.

In particular, this provides an explicit $2\times 2$ witness that the optimal constant in
Problem~\ref{prob:schatten-sym-vs-sym-p} satisfies $c_2^{\mathrm{sym}}(m,n)>1$ for every $n\ge 2$ (by embedding the
$2\times2$ example into $\mathbb M_n$ as a direct summand).
	\begin{eg}\label{ex:c2gt1}
		Let
		\[
		A=\begin{pmatrix}0&-1\\[2pt]1&-4\end{pmatrix},
		\qquad
		B=\begin{pmatrix}-16&-7\\[2pt]9&4\end{pmatrix}.
		\]
		Then a direct numerical evaluation gives
		\[
		\frac{\bigl\||A+B|_{\mathrm{sym}}\bigr\|_2}{\bigl\||A|_{\mathrm{sym}}+|B|_{\mathrm{sym}}\bigr\|_2}
		\approx 1.0144\ >\ 1.
		\]
		Hence $c_2^{\mathrm{sym}}(m,2)>1$, and therefore $c_2^{\mathrm{sym}}(m,n)>1$ for every $n\ge2$ (by embedding $2\times2$
		matrices into $\mathbb M_n$ as a direct summand).
	\end{eg}
	
	Next, we state some	endpoint values and general bounds for $c_p^{\mathrm{qsym}}(m,n)$ in Problem~\ref{prob:schatten-qsym-vs-qsym-p}.
	
		\begin{thm}
		Let $A_1,\dots,A_m\in\mathbb M_n$, and let $1\le p\le\infty$. Let $c_p^{\mathrm{qsym}}(m,n)$ to be the smallest
		constant such that
		\[
		\Bigl\lVert \Bigl|\sum_{k=1}^m A_k\Bigr|_{\mathrm{qsym}} \Bigr\rVert_p
		\ \le\ c_p^{\mathrm{qsym}}(m,n)\,
		\Bigl\lVert \sum_{k=1}^m |A_k|_{\mathrm{qsym}} \Bigr\rVert_p.
		\]
		Then the following statements hold:
\begin{enumerate}
	\item $c_{1}^{\mathrm{qsym}}(m,n)=1$;
	\item $c_{2}^{\mathrm{qsym}}(m,n)\le \sqrt{\frac{1+\sqrt{\min\{m,2n\}}}{2}}$;
	\item $c_{\infty}^{\mathrm{qsym}}(m,n)\le \sqrt{\min\{m,2n\}}$;
	\item $c_{p}^{\mathrm{qsym}}(m,n)\le \bigl(\min\{m,2n\}\bigr)^{\frac12-\frac{1}{2p}}$ for all $1\le p\le\infty$.
\end{enumerate}
	\end{thm}
	\begin{proof}
			For $A\in\mathbb M_n$, define the $2n\times 2n$ matrix
		\[
		\Phi(A):=\frac1{\sqrt2}\begin{pmatrix}A&0\\ A^*&0\end{pmatrix}.
		\]
		A direct computation shows
		\begin{equation}\label{eq:Phi-abs-p}
			|\Phi(A)|=\begin{pmatrix}|A|_{\mathrm{qsym}}&0\\0&0\end{pmatrix}.
		\end{equation}
		Clearly, $	\Phi$ is additive, that is,
		\begin{equation}\label{eq:Phi-additive-p}
			\Phi(\sum_{k=1}^m A_k)=\sum_{k=1}^m \Phi(A_k).
		\end{equation}
		
		Let $c_p^{\mathrm{abs}}(m,n)$ be the optimal constant in Problem~\ref{prob:general-lee-problem}.
		Applying it to the matrices $\Phi(A_1),\dots$, $\Phi(A_m)\in\mathbb M_{2n}$ gives
		\begin{equation}\label{eq:Phi-Lee-inequality-p}
			\Bigl\|\sum_{k=1}^m \Phi(A_k)\Bigr\|_p
			\le c_p^{\mathrm{abs}}(m,2n) \Bigl\|\sum_{k=1}^m |\Phi(A_k)|\Bigr\|_p.
		\end{equation}
		Thus, 
		\begin{align*}
			\left\| \begin{pmatrix}|\sum_{k=1}^m A_k|_{\mathrm{qsym}}&0\\0&0\end{pmatrix}\right\|_p &=
			\bigl\||\Phi(\sum_{k=1}^m A_k)|\bigr\|_p&\text{(By \eqref{eq:Phi-abs-p})}\\
			&=	\bigl\|\Phi(\sum_{k=1}^m A_k)\bigr\|_p\\
			&=
			\bigl\|\sum_{k=1}^m \Phi(A_k)\bigr\|_p&\text{(By \eqref{eq:Phi-additive-p})}\\
			&\le c_p^{\mathrm{abs}}(m,2n)  \Bigl\|\sum_{k=1}^m |\Phi(A_k)|\Bigr\|_p&\text{(By \eqref{eq:Phi-Lee-inequality-p})}\\
			&=c_p^{\mathrm{abs}}(m,2n)  \Bigl\|\sum_{k=1}^m
			\begin{pmatrix}|A_k|_{\mathrm{qsym}}&0\\0&0\end{pmatrix}\Bigr\|_p &\text{(By \eqref{eq:Phi-abs-p})}\\
			&=c_p^{\mathrm{abs}}(m,2n)  \Bigl\|
			\begin{pmatrix}\sum_{k=1}^m |A_k|_{\mathrm{qsym}}&0\\0&0\end{pmatrix}\Bigr\|_p.
		\end{align*}

		Finally, for any unitarily invariant norm, adjoining a zero direct summand does not change the norm
		(since it only appends zero singular values). Therefore,
		\[
		\left\| \begin{pmatrix}|\sum_{k=1}^m A_k|_{\mathrm{qsym}}&0\\0&0\end{pmatrix}\right\|_p
		=\bigl\||\sum_{k=1}^m A_k|_{\mathrm{qsym}}\bigr\|_p,
		\qquad
		\Bigl\|
		\begin{pmatrix}\sum_{k=1}^m |A_k|_{\mathrm{qsym}}&0\\0&0\end{pmatrix}\Bigr\|_p
		=\Bigl\|\sum_{k=1}^m |A_k|_{\mathrm{qsym}}\Bigr\|_p.
		\]
		Combining these inequalities yields
		\[
		\Bigl\lVert \Bigl|\sum_{k=1}^m A_k\Bigr|_{\mathrm{qsym}} \Bigr\rVert_p
		\le c_p^{\mathrm{abs}}(m,2n) \,
		\Bigl\lVert \sum_{k=1}^m \bigl|A_k\bigr|_{\mathrm{qsym}} \Bigr\rVert_p.
		\]
		Thus, $c_p^{\mathrm{qsym}}(m,n)\le c_p^{\mathrm{abs}}(m,2n)$. Applying partial results on Problem~\ref{prob:general-lee-problem} yields our desired result.
	\end{proof}

	\section*{Declaration of competing interest}
	The author declares no competing interests.
	
	\section*{Data availability}
	No data was used for the research described in the article.
	
	\section*{Acknowledgments}
	Teng Zhang is supported by the China Scholarship Council, the Young Elite Scientists Sponsorship Program for PhD Students (China Association for Science and Technology), and the Fundamental Research Funds for the Central Universities at Xi'an Jiaotong University (Grant No.~xzy022024045).

\end{document}